\documentclass[11pt]{article}
\usepackage{epsfig,mst-stylefile,amssymb,amsmath,harvard}

\usepackage[svgnames]{xcolor}

\usepackage{xr}

\newcommand{\bbZ}{{\Bbb Z}}
\newcommand{\bbR}{{\Bbb R}}
\newcommand{\bbN}{{\Bbb N}}

\newcommand{\bbP}{{\Bbb P}}

\newcommand{\bbE}{{\Bbb E}}
\newcommand{\E}{{\Bbb E}}

\newcommand{\diag}{\textnormal{diag}}

\DeclareMathOperator*{\plim}{\mathit{p}-lim}

\newcommand{\vecoper}{\textnormal{vec}}
\usepackage{xcolor}

\newcommand{\W}{\mathbf{W}}
\newcommand{\B}{\mathbf{B}}
\newcommand{\p}{\mathbf{p}}

\usepackage{xcolor}

\renewcommand{\cite}{\citeyear}

\graphicspath{{figures_highdim_main/}}

\begin{document}

\title{Wavelet eigenvalue regression in high dimensions
\thanks{P.A.\ and H.W.\ were partially supported by ANR-16-CE33-0020 MultiFracs, France. H.W.\ was also partially supported by ANR-18-CE45-0007 MUTATION. G.D.'s long term visits to ENS de Lyon were supported by the school, the CNRS and the Simons Foundation collaboration grant $\#714014$. The authors also gratefully acknowledge the support and resources from the Center for High Performance Computing at the University of Utah as well as the high performance computing (HPC) resources and services provided by Technology Services at Tulane University.}
\thanks{The authors would like to thank an anonymous reviewer whose comments and suggestions improved the manuscript.}
}

\author{Patrice Abry \\ Univ Lyon, ENS de Lyon,\\ Univ Claude Bernard,
CNRS, \\ Laboratoire de Physique, \\ F-69342 Lyon, France
\and   B.\ Cooper Boniece \\  Department of Mathematics\\ University of Utah
\and   Gustavo Didier \\ Mathematics Department\\ Tulane University
\and   Herwig Wendt\\ IRIT-ENSEEIHT, CNRS (UMR 5505),\\ Universit\'{e} de Toulouse, France}

\maketitle

\begin{abstract}
In this paper, we construct the wavelet eigenvalue regression methodology (Abry and Didier~\cite{abry:didier:2018:dim2,abry:didier:2018:n-variate}) in high dimensions. We assume that possibly non-Gaussian, finite-variance $p$-variate measurements are made of a low-dimensional $r$-variate ($r \ll p$) fractional stochastic process with non-canonical scaling coordinates and in the presence of additive high-dimensional noise. The measurements are correlated both time-wise and between rows. Building upon the asymptotic and large scale properties of wavelet random matrices in high dimensions, the wavelet eigenvalue regression is shown to be consistent and, under additional assumptions, asymptotically Gaussian in the estimation of the fractal structure of the system. We further construct a consistent estimator of the effective dimension $r$ of the system that significantly increases the robustness of the methodology. The estimation performance over finite samples is studied by means of simulations.
\end{abstract}

\section{Introduction}

A \textit{wavelet} is a unit $L^2(\bbR)$-norm function that annihilates polynomials. For a fixed (octave) $j \in \bbN$, a \textit{wavelet random matrix} is given by
\begin{equation}\label{e:W(a2^j)_intro}
{\mathbf W}(a(n)2^j) = \frac{1}{n_{a,j}}\sum^{n_{a,j}}_{k=1} D(a(n)2^j,k)D(a(n)2^j,k)^*.
\end{equation}
In \eqref{e:W(a2^j)_intro}, $^*$ denotes transposition, $n_{a,j} = n/a(n)2^j$ is the number of wavelet-domain observations for a sample size $n$, and each random vector $D(a(n)2^j,k)$ is the wavelet transform of a multivariate stochastic process $Y$ at the dyadic scale $a(n)2^j$ and shift $k \in \bbZ$. The entries of $\{D(a(n)2^j,k)\}_{k \in \bbN}$ are generally correlated. The so-named \textit{wavelet eigenanalysis} methodology consists in using the behavior across scales of the eigenvalues of wavelet random matrices to study the fractality of stochastic systems (Abry and Didier \cite{abry:didier:2018:dim2,abry:didier:2018:n-variate}). In this paper, we build upon recent results on the properties of large wavelet random matrices (Abry et al.\ \cite{abry:boniece:didier:wendt:2022:wavelet_eigenanalysis}) to construct a wavelet eigenvalue regression methodology in high dimensions. The (possibly non-Gaussian) underlying stochastic process $Y = \{Y(t)\}_{t \in \bbZ}$ is assumed to have the form
\begin{equation}\label{e:Y(t)}
Y(t) = P X(t) + Z(t), \quad t \in \bbZ.
\end{equation}
In \eqref{e:Y(t)}, both $Y$ and the noise term $Z = \{Z(t)\}_{t \in \bbZ}$ are (high-dimensional) $p=p(n)$-variate processes, $P=P(n)$ is a rectangular coordinates matrix and, for fixed $r \in \bbN$, $X = \{X(t)\}_{t \in \bbZ}$ is a (low-dimensional) $r$-variate fractional process. In particular, the measurements $Y$ are correlated time-wise and between rows. We show that, if the ratio $a(n)p(n)/n$ converges to a positive constant, the wavelet eigenvalue regression provides a consistent and, under additional assumptions, asymptotically Gaussian estimator of the underlying low-dimensional fractal structure of the system. In addition, we construct a consistent estimator of the effective dimension $r$ of the system that significantly increases the robustness of the statistical methodology. The performance of the statistical protocols over finite samples is further studied by means of simulations.

Since the 1950s, the spectral behavior of large dimensional random matrices has attracted considerable attention from the mathematical research community. In quantum mechanics, for example, random matrices are of great interest as statistical mechanical models of infinite dimensional and possibly unknown Hamiltonian operators (e.g., Mehta and Gaudin \cite{mehta:gaudin:1960}, Dyson \cite{dyson:1962}, Arous and Guionnet \cite{arous:guionnet:1997}, Soshnikov \cite{soshnikov:1999}, Mehta \cite{mehta:2004}, Deift \cite{deift:2007}, Anderson et al.\ \cite{anderson:guionnet:zeitouni:2010}, Erd\H{o}s et al.\ \cite{erdos:yau:yin:2012}). Random matrices have also naturally emerged as one essential mathematical framework for the modern era of ``Big Data" (Briody \cite{briody:2011}), when hundreds to several tens of thousands of time series get recorded and stored on a daily basis. In coping with the data deluge, one is often interested in understanding the behavior of random constructs such as the spectral distribution of sample covariance matrices when the dimension $p$ is comparable to the sample size $n$, including instances with dependent measurements (e.g., Tao and Vu \cite{tao:vu:2012}, Paul and Aue \cite{paul:aue:2014}, Basu and Michailidis \cite{basu:michailidis:2015}, Giraud \cite{giraud:2015}, Yao et al.\ \cite{yao:zheng:bai:2015},  Chakrabarty et al.\ \cite{chakrabarty:hazra:sarkat:2016}, Merlev\`{e}de and Peligrad \cite{merlevede:peligrad:2016}, Che \cite{che:2017}, Taylor and Salhi \cite{taylor:salhi:2017}, Wang et al.\ \cite{wang:aue:paul:2017}, Zhang and Wu \cite{zhang:wu:2017}, Erd\H{o}s et al.\ \cite{erdos:kruger:schroder:2019}, Horv\'{a}th and Rice \cite{horvath:rice:2019}, Merlev\`{e}de et al.\ \cite{merlevede:najim:tian:2019}, Wainwright \cite{wainwright:2019}, Bourguin et al.~\cite{bourguin:diez:tudor:2021}).

In turn, \textit{scale invariance} manifests itself in a wide range of natural and social phenomena such as in climate studies (Isotta et al.\ \cite{isotta:etal:2014}), critical phenomena (Sornette \cite{sornette:2006}), dendrochronology (Bai and Taqqu~\cite{bai:taqqu:2018}), hydrology (Benson et al.\ \cite{benson:baeumer:scheffler:2006}) and turbulence (Kolmogorov~\cite{Kolmogorovturbulence}). In a multidimensional setting, scaling behavior does not always appear along standard coordinate axes, and often involves multiple (scaling) relations. A $\bbR^r$-valued stochastic process $X$ is called \textit{operator self-similar} (o.s.s.; Laha and Rohatgi \cite{laha:rohatgi:1981}, Hudson and Mason \cite{hudson:mason:1982}) if it exhibits the scaling property
\begin{equation}\label{e:def_ss}
\{X(ct)\}_{t\in\bbR} \stackrel {\textnormal{f.d.d.}}{=} \{c^H X(t)\}_{t\in\bbR}, \quad c>0.
\end{equation}
In \eqref{e:def_ss}, $H$ is some (Hurst) matrix $H$ whose eigenvalues have real parts lying in the interval $(0,1]$ and $c^H := \exp\{\log(c) H\} = \sum^{\infty}_{k=0} \frac{(\log(c) H)^k}{k!}$. A canonical model for multivariate fractional systems is operator fractional Brownian motion (ofBm), namely, a Gaussian, o.s.s., stationary-increment stochastic process (Maejima and Mason \cite{maejima:mason:1994}, Mason and Xiao \cite{mason:xiao:2002}, Didier and Pipiras \cite{didier:pipiras:2012}). In particular, ofBm is the natural multivariate generalization of the classical fBm (Mandelbrot and Van Ness \cite{mandelbrot:vanness:1968}).

In the characterization of scaling properties, the use of eigenanalysis was first proposed in Meerschaert and Scheffler \cite{meerschaert:scheffler:1999,meerschaert:scheffler:2003} and Becker-Kern and Pap \cite{becker-kern:pap:2008}. It has also been used in the cointegration literature (e.g., Phillips and Ouliaris \cite{phillips:ouliaris:1988}, Li et al.\ \cite{li:pan:yao:2009}). In Abry and Didier \cite{abry:didier:2018:n-variate,abry:didier:2018:dim2}, \textit{wavelet eigenanalysis} is proposed in the construction of a general methodology for the statistical identification of the scaling (Hurst) structure of ofBm in low dimensions.

Wavelet random matrices were used in high-dimensional modeling contexts first in Abry et al.\ \cite{abry:wendt:didier:2018:detecting_highdim} and Boniece et al.\ \cite{boniece:wendt:didier:abry:2019}. In Abry et al.\ \cite{abry:boniece:didier:wendt:2022:wavelet_eigenanalysis}, the fundamental mathematical properties of the eigenvalues $\lambda_{1}\big({\mathbf W}(a(n)2^j)\big) \leq \hdots \leq \lambda_{p(n)}\big({\mathbf W}(a(n)2^j)\big)$ of large wavelet random matrices are, to the best of our knowledge, for the first time considered. In this paper, we build upon such fundamental properties to construct robust statistical methodology in high dimensions. The measurements are assumed to be of the form \eqref{e:Y(t)}, where the fractional behavior of $X$ is characterized by a scaling matrix of the Jordan form
\begin{equation}\label{e:H=PHdiag(h1,...,hn)P^(-1)H}
H = P_H \hspace{0.5mm}\textnormal{diag}(h_1,\hdots,h_r)\hspace{0.5mm}P^{-1}_{H},
\end{equation}
with real eigenvalues. The measurements $Y$ display correlation time-wise and between rows. In applications, models of the form \eqref{e:Y(t)} and related models appear, for example, in neuroscience, fMRI imaging and signal processing (Ciuciu et al.\ \cite{ciuciu:varoquaux:abry:sadaghiani:kleinschmidt:2012}, Liu et al.\ \cite{liu:aue:paul:2015}; cf.\ Chauduri et al.\ \cite{chaudhuri:gercek:pandey:peyrache:fiete:2019}, Stringer et al.\ \cite{stringer:pachitariu:steinmetz:carandini:harris:2019}) and in econometrics (e.g., Brown \cite{brown:1989}, Zhang et al.\ \cite{zhang:robinson:yao:2019}). The goal of statistical inference is to characterize the low-dimensional ($r$-variate) fractal structure in high-dimensional data. This requires robustness not only with respect to non-canonical scaling coordinates, as in fixed dimensions (the so-called \textit{amplitude }and \textit{dominance effects}; see Abry and Didier \cite{abry:didier:2018:dim2}), but also to high-dimensional environmental noise. Starting from measurements $Y$, we propose a multiscale wavelet eigenvalue regression estimator
$$
\{ \widehat{\ell}_{i} \}_{i = 1,\hdots,p} :=  \Big\{\frac{1}{2}\Big(\sum_{j=j_1}^{j_2} w_j \log_2 \lambda_{i}\big({\mathbf W}(a(n)2^j)\big) -1\Big)\Big\}_{i = 1,\hdots,p},
$$
for fixed wavelet octaves $j_1$, $j_2$ and particular weights $w_j$ (see Definition \ref{def:eigenvalue_estimator}). We show that the fixed-dimensional subvector
\begin{equation}\label{e:h_q=ell_i}
\{\widehat{h}_{q}\}_{q=1,\hdots,r} := \{ \widehat{\ell}_{i} \}_{i = p(n)-r+1,\hdots,p(n)}
\end{equation}
is a consistent and, under additional assumptions, asymptotically Gaussian vector estimator of the scaling eigenvalue structure $\{h_q\}_{q=1,\hdots,r}$ of $X$ (Theorems \ref{t:h-hat_consistency} and \ref{t:h-hat_asympt_normality}; see also Figure \ref{fig:ellhat}). The asymptotic properties of the estimator are grounded in the high-dimensional behavior of the eigenvalues of wavelet random matrices. In fact, as shown in Abry et al.\ \cite{abry:boniece:didier:wendt:2022:wavelet_eigenanalysis}, as $n$ and $a(n)$ grow the $r$ largest wavelet eigenvalues satisfy scaling relations of the form
\begin{equation}\label{e:lambda_p-r+q_scales}
\lambda_{p(n)-r+q}\big({\mathbf W}(a(n)2^j)\big) \stackrel{\bbP}\sim \xi_q(2^j) \hspace{0.5mm}a(n)^{2h_q + 1}, \quad q=1,\hdots,r,
\end{equation}
for deterministic $\xi_q(2^j)$. Also, under additional assumptions, they display asymptotically Gaussian fluctuations (cf.\ Theorems \ref{t:lim_n_a*lambda/a^(2h+1)} and \ref{t:asympt_normality_lambdap-r+q}). The high-dimensional statistical analysis requires taking the three-way limit
\begin{equation}\label{e:three-fold_lim}
\lim_{n \rightarrow \infty} \frac{p(n) a(n)}{n} =: c \in [0,\infty).
\end{equation}
For the sake of illustration and comparison, note that traditional analysis of large sample covariance matrices usually involves the ratio $\lim_{n \rightarrow \infty} p(n)/n$ and the largest eigenvalue often exhibits universality in the form of Tracy--Widom fluctuations (Bai and Silverstein \cite{bai:silverstein:2010}, Lee and Schnelli \cite{lee:schnelli:2016}). Vis-\`{a}-vis Abry et al.\ \cite{abry:boniece:didier:wendt:2022:wavelet_eigenanalysis}, the main mathematical contribution of this paper stems from the fact that the wavelet eigenvalue regression demands careful analysis of the rate of convergence of the rescaled $r$ largest eigenvalues of the deterministic wavelet matrices $\bbE \W (a(n)2^j)$. This involves studying the rate of angular convergence of the associated (deterministic) wavelet eigenvectors in terms of the high-dimensional coordinates $P = P(n)$ (see Proposition \ref{p:|lambdaq(EW)-xiq(2^j)|_bound_2}).

Since \eqref{e:h_q=ell_i} presupposes knowledge of the dimension $r$ of the hidden fractional process $X$, we further put forth a wavelet eigenvalue regression-based estimator $\widehat{r}$ of $r$ (see Definition \ref{def:r-hat} and Theorem \ref{t:r-hat->r}). The construction is based on the fact that, by contrast with \eqref{e:lambda_p-r+q_scales}, the lower eigenvalues $\lambda_{\ell}\big({\mathbf W}(a(n)2^j)\big)$, $\ell = 1,\hdots,p(n)-r$, of wavelet random matrices are bounded in probability (cf.\ Theorem \ref{t:lim_n_a*lambda/a^(2h+1)}). The use of the estimator $\widehat{r}$ greatly increases the robustness of the overall statistical methodology over finite samples by providing a quantitative measure of separation between the scaling behaviors of large and small wavelet log-eigenvalues (cf.\ Figure \ref{fig:ellhat}; for effective dimension estimation in different contexts, see, for example, Nadakuditi and Edelman \cite{nadakuditi:edelman:2008}, Little et al.\ \cite{little:lee:jung:maggioni:2009}, Lam and Yao \cite{lam:yao:2012}).

As in Abry et al.\ \cite{abry:boniece:didier:wendt:2022:wavelet_eigenanalysis}, for the sake of clarity and mathematical generality, the assumptions are stated directly in the wavelet domain (Section \ref{s:framework}). In particular, the measurements $Y$ are possibly non-Gaussian. A brief illustration of instances covered by the assumptions is provided in Examples \ref{ex:Prop_4_1_in_Abry_et_al}--\ref{ex:Haar_sub-Gaussian}.

We provide broad computational studies that demonstrate the convergence to Gaussianity of $\{\widehat{h}_{q}\}_{q=1,\hdots,r}$ in the high dimensional limit for various instances of $n_{a,j}$, $p$ and $r$. The experiments confirm that convergence takes place regardless of the (fixed) ratio $p/n_{a,j}$. We further use computational experiments to study the optimal thresholding procedure involved in the use of $\widehat{r}$.

This paper is organized as follows. In Section \ref{s:framework}, we provide the basic wavelet framework, definitions and wavelet-domain assumptions used throughout the paper. In Section \ref{s:main}, we state and discuss the main results on the wavelet eigenvalue regression. In Section \ref{s:MC}, we display and discuss the broad simulation studies. In Section \ref{s:conclusion}, we lay out conclusions and discuss several open problems that this work leads to. All proofs can be found in the appendix.

\begin{center}
\begin{figure}[h!]
\centering
\setlength{\tabcolsep}{2pt}
\begin{tabular}{ccc}
& $\widehat{\ell}_i,\;p/n_{j_2} = 1/2$ & $\widehat{\ell}_i,\;p/n_{j_2} = 1/4$ \\
\rotatebox[origin=t]{90}{local $\widehat{\ell}_i$\hskip-60mm}&\includegraphics[width=0.45\linewidth, trim=50 20 50 25, clip]{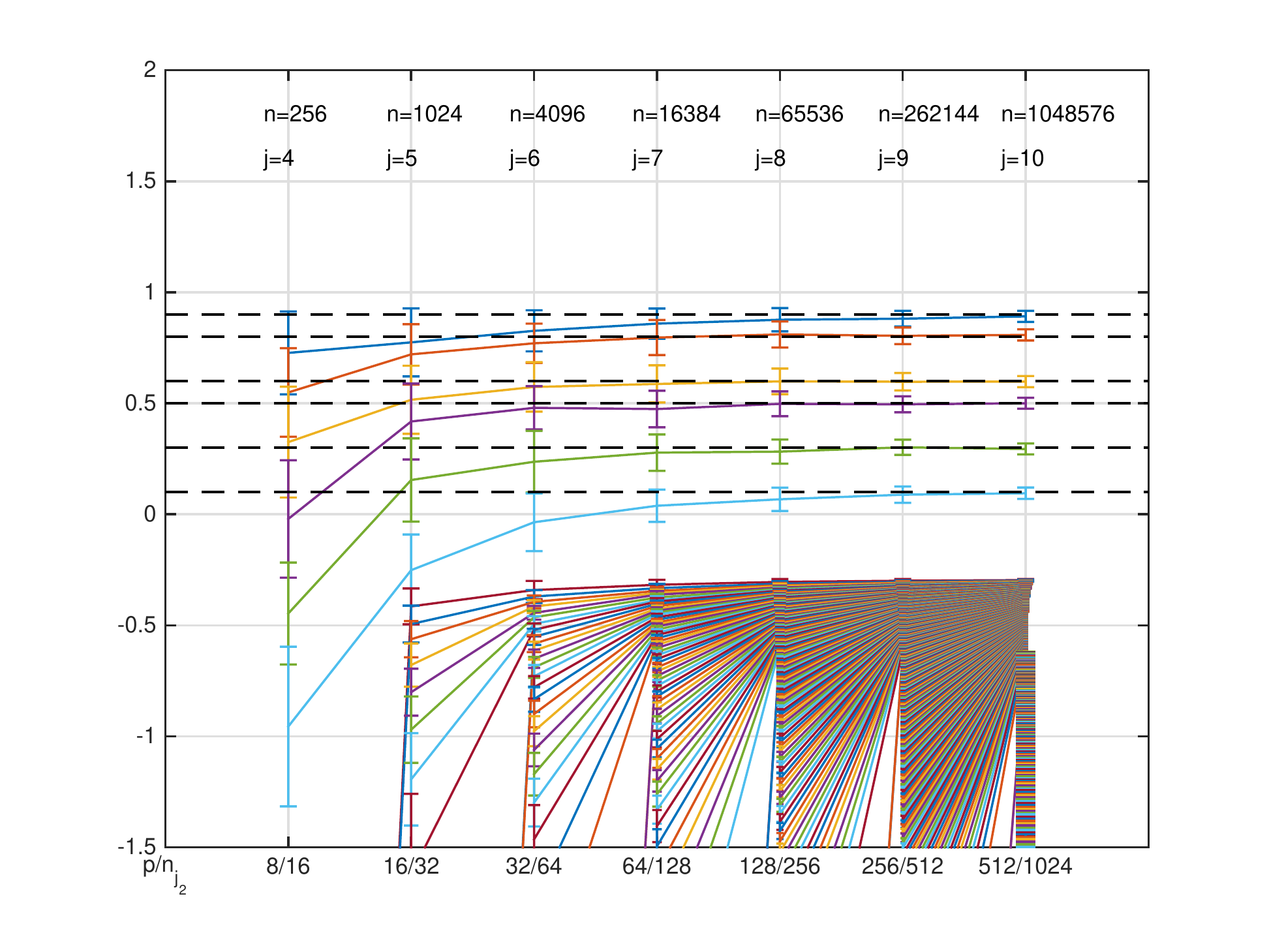}
&\includegraphics[width=0.45\linewidth, trim=50 20 50 25, clip]{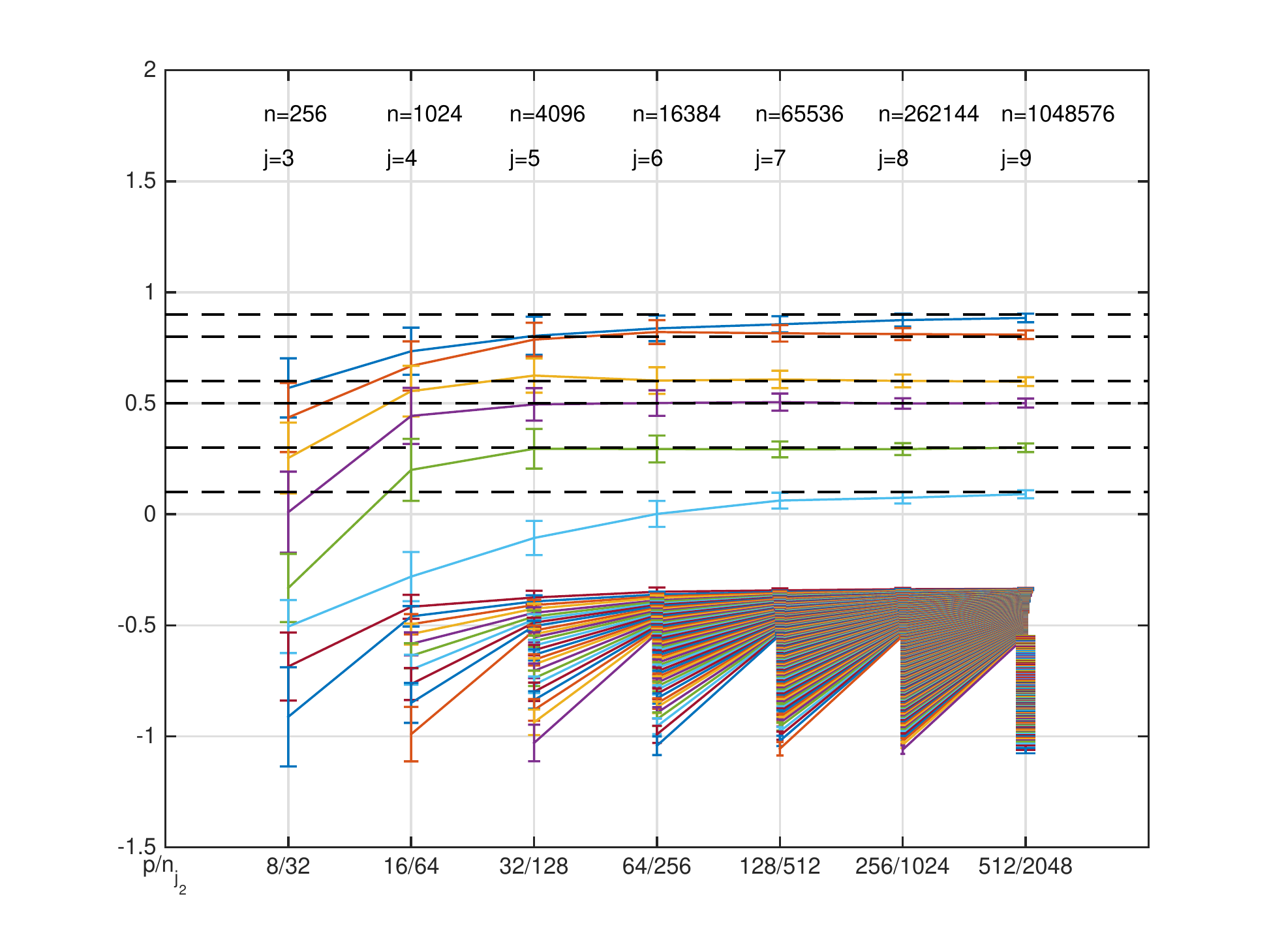}\\
& $\widehat{\ell}_i,\;p/n_{j_2} = 1/2$ & $\widehat{\ell}_i,\;p/n_{j_2} = 1/4$ \\
\rotatebox[origin=t]{90}{local $\widehat{\ell}_i$\hskip-60mm}&\includegraphics[width=0.45\linewidth, trim=50 20 50 25, clip]{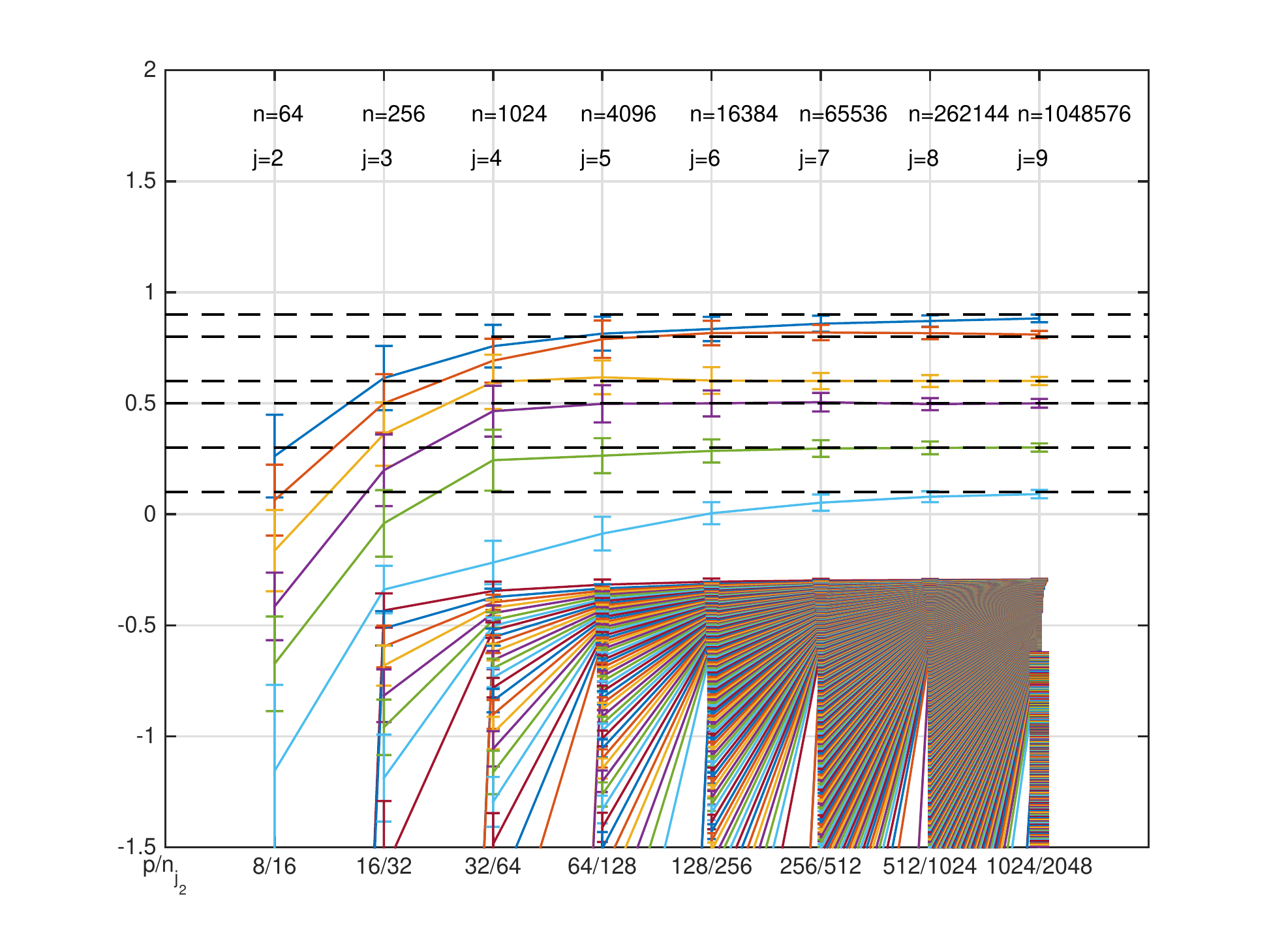}
&\includegraphics[width=0.45\linewidth, trim=50 20 50 25, clip]{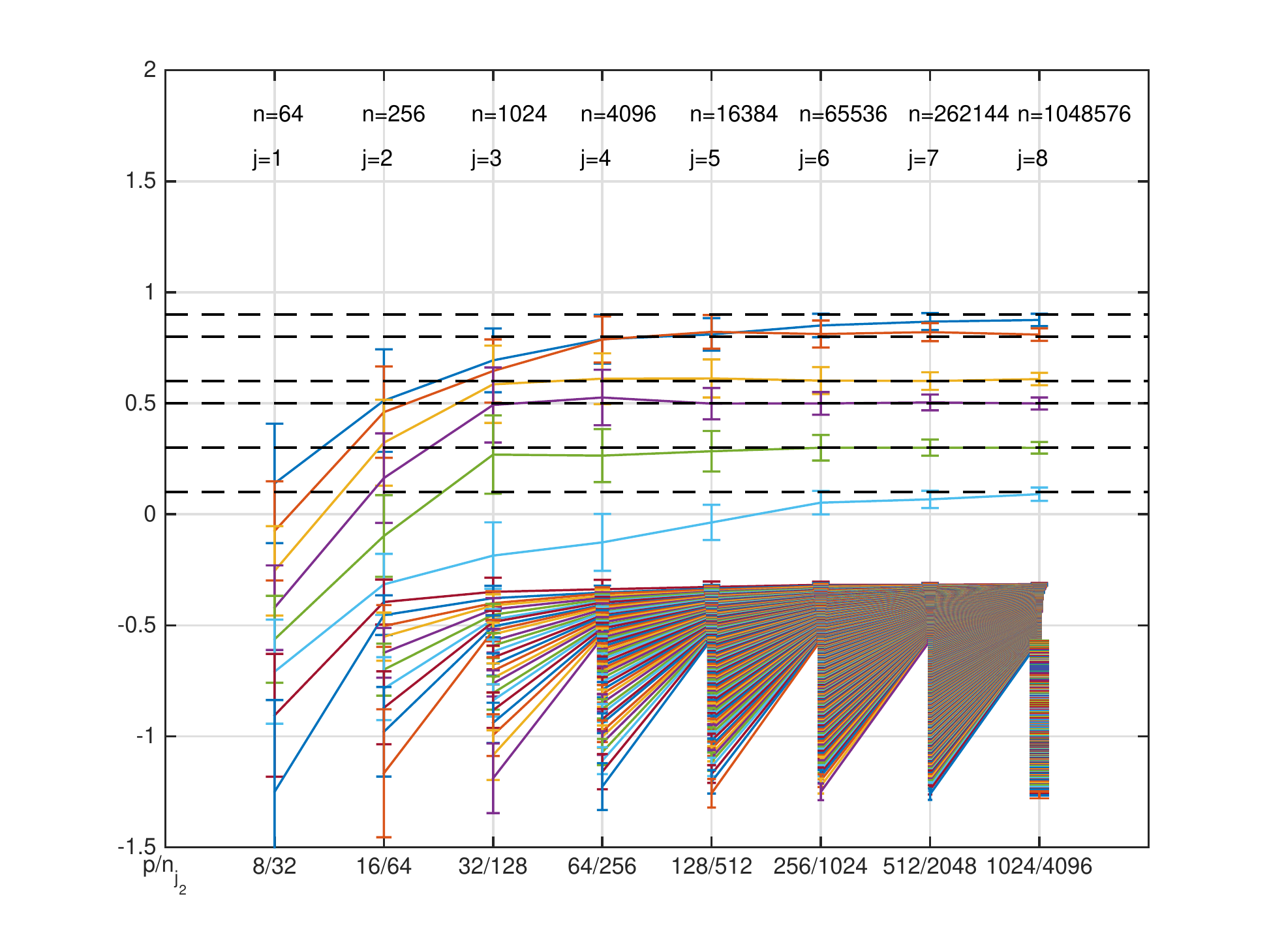}\\
\end{tabular}
\caption[The high-dimensional asymptotic behavior of $\widehat \ell_i$]{\label{fig:ellhat}{\textbf{The high-dimensional asymptotic behavior of $\widehat \ell_i$.}} The plots above show the asymptotic behavior of local estimates $\widehat{\ell}_q$ around each $j$ as $p$ and $n_j = n/2^j$ increase with their ratio fixed at ${p/n_j}= c$, with $c=1/2$ (left column) and $c=1/4$ (right column).  There are $r=6$ scaling eigenvalues $h_i\in (0.1,0.3, 0.5,0.6,0.8,0.9)$. For each $j$, the reported values $\widehat\ell_i$ are based on the octave range $(j_1,j_2)=(j-1,j+1)$. Note the $\widehat\ell_i$ associated with non-scaling eigenvalues stays below $0$ as $j$ increases in all instances, and the $\widehat{\ell}_q$ associated with the $r=6$ scaling eigenvalues approach their theoretical values.  With $c$ fixed (along each column), the larger sample size $n$ (bottom row) results in near-convergence at smaller $j$ (note the axes have been shifted to align the curves).  Likewise, as the $c$ decreases (between left and right columns), the near-convergence also occurs at smaller octaves $j$.  Compared to the wavelet log-eigenvalues themselves (not shown), the multiscale estimator $\widehat\ell_i$ displays reduced bias.}%
\end{figure}
\end{center}

\section{Framework and notation}\label{s:framework}

In this section, we describe the notation, basic wavelet framework and wavelet-domain assumptions used throughout the paper. Note that this is the same wavelet framework put forth in Abry et al.\ \cite{abry:boniece:didier:wendt:2022:wavelet_eigenanalysis}, Section 2.

For $m \in \bbN$, let $M(m,\bbR)$ %
be the space of $m \times m$ real-valued matrices, and %
 $M(m_1,m_2,\bbR)$ be the spaces of $m_1 \times m_2$ real-valued matrices. Let ${\mathcal S}(m,\bbR)$, ${\mathcal S}_{\geq 0}(m,\bbR)$ and ${\mathcal S}_{> 0}(m,\bbR)$, respectively, be the space of $m \times m$ symmetric matrices and the cones of $m\times m$ positive definite and positive semidefinite symmetric matrices.  %
Throughout the manuscript, $\|M\|$ denotes the operator norm of a matrix $M \in M(p,\bbR)$ in arbitrary dimension $p$, i.e.,
$\|M\| = \sqrt{\sup_{{\mathbf u} \in S^{p-1}} {\mathbf u}^* M M^* {\mathbf u}} = \sqrt{\sup_{{\mathbf u} \in S^{p-1}} {\mathbf u}^* M^* M {\mathbf u}}$. For any $M \in {\mathcal S}(p,\bbR)$,
\begin{equation}\label{e:lambda1(M)=<...=<lambdap(M)}
- \infty < \lambda_1(M) \leq \hdots \leq \lambda_q(M) \leq \hdots \leq \lambda_p(M) < \infty
\end{equation}
denotes the set of ordered eigenvalues of the matrix $M$. For $S = (s_{i_1,i_2})_{i_1,i_2=1,\dots,n}\in M(n,\mathbb{R})$, let
\begin{equation}\label{e:vec_definitions}
\textnormal{vec}_{{\mathcal S}}(S)=(s_{11},s_{21},\dots,s_{n1},s_{22},s_{32},\dots,s_{n2},\dots,s_{nn}).
\end{equation}
We use the asymptotic notation
\begin{equation}\label{e:OP(1),OmegaP(1)}
o_\bbP(1), \quad O_\bbP(1)%
\end{equation}
to describe sequences of random vectors or matrices whose operator norms vanish and are bounded above, respectively, in probability.

\subsection{Wavelet analysis}\label{s:wavelet_analysis}

Throughout the paper, we make use of a wavelet multiresolution analysis (MRA; see Mallat \cite{mallat:1999}, chapter 7), which decomposes $L^2(\mathbb{R})$ into a sequence of \textit{approximation} (low-frequency) and \textit{detail} (high-frequency) subspaces $V_j$ and $W_j$, respectively, associated with different scales of analysis $2^{j}$. In almost all mathematical statements, we make assumptions $(W1-W4)$ on the underlying wavelet MRA. Such assumptions are accurately described in Section \ref{s:MRA_assumptions}. In particular, we make use of a compactly supported wavelet basis. %

So, let $\phi$ and $\psi$ be the scaling and wavelet functions, respectively, associated with the wavelet MRA. We further suppose the wavelet coefficients stem from Mallat's pyramidal algorithm (Mallat \cite{mallat:1999}, chapter 7). For expositional simplicity, in our description of the algorithm we use the $\bbR^p$-valued process $Y$, though analogous developments also hold for both $X$ and $Z$. Initially, suppose an infinite time series
\begin{equation}\label{e:infinite_sample}
\{Y(k)\}_{k \in \bbZ},
\end{equation}
associated with the starting scale $2^j = 1$ (or octave $j = 0$), is available. Then, we can apply Mallat's algorithm to extract the so-named \textit{approximation} $(A(2^{j+1},\cdot))$ and \textit{detail} $(D(2^{j+1},\cdot))$ coefficients at coarser scales $2^{j+1}$ by means of an iterative procedure. In fact, as commonly done in the wavelet literature, we initialize the algorithm with the process
\begin{equation}\label{e:Btilde}
\bbR^p \ni \widetilde{Y}(t) := \sum_{k \in \bbZ} Y(k)\phi(t-k), \quad  t\in \bbR.
\end{equation}
By the orthogonality of the shifted scaling functions $\{\phi(\cdot - k)\}_{k \in \bbZ}$,
\begin{equation}\label{e:a(0,k)}
\bbR^p \ni  A(2^0,k)= \int_\bbR \widetilde{Y}(t)\phi(t-k)dt= Y(k), \quad k \in \bbZ
\end{equation}
(see Stoev et al.\ \cite{stoev:pipiras:taqqu:2002}, proof of Lemma 6.1, or Moulines et al.\ \cite{moulines:roueff:taqqu:2007:JTSA}, p.\ 160; cf.\ Abry and Flandrin \cite{abry:flandrin:1994}, p.\ 33). In other words, the initial sequence, at octave $j = 0$, of approximation coefficients is given by the original time series. To obtain approximation and detail coefficients at coarser scales, we use Mallat's iterative procedure
\begin{equation}\label{e:Mallat}
A(2^{j+1},k) = \sum_{k'\in \mathbb{Z}} u_{k'-2k}A(2^j,k'),\quad D(2^{j+1},k) =\sum_{k'\in \mathbb{Z}}v_{k'-2k} A(2^j,k'), \quad j \in \mathbb{N} \cup \{0\}, \quad k \in \mathbb{Z},
\end{equation}
where the (scalar) filter sequences $\{ u_k:=2^{-1/2}\int_\bbR \phi(t/2)\phi(t-k)dt \}_{k\in\bbZ}$, $\{v_k:=2^{-1/2}\int_\bbR\psi(t/2)\phi(t-k)dt\}_{k\in\bbZ}$ are called low- and high-pass MRA filters, respectively. Due to the assumed compactness of the supports of $\psi$ and of the associated scaling function $\phi$ (see condition \eqref{e:supp_psi=compact}), only a finite number of filter terms is nonzero, which is convenient for computational purposes (Daubechies \cite{daubechies:1992}). {Hereinafter, we assume without loss of generality that $\text{supp}(\phi) = \text{supp}(\psi)=[0,T]$ (cf.\ Moulines et al \cite{moulines:roueff:taqqu:2007:JTSA}, p.\ 160).} Moreover, the wavelet (detail) coefficients $D(2^j,k)$ of $Y$ can be expressed as
\begin{equation}\label{e:disc2}
\bbR^p \ni D(2^j,k) = \sum_{\ell\in\bbZ} Y(\ell)h_{j,2^jk-\ell},
\end{equation}
where the filter terms are defined by
\begin{equation}\label{e:hj,l}
\bbR \ni h_{j,\ell} =2^{-j/2}\int_\bbR\phi(t+\ell)\psi(2^{-j}t)dt.
\end{equation}
If we replace \eqref{e:infinite_sample} with the realistic assumption that only a finite length time series
\begin{equation}\label{e:finite_sample}
\{Y(k)\}_{k=1,\hdots,n}
\end{equation}
is available, writing $\widetilde{Y}^{(n)}(t) := \sum_{k=1}^{n}Y(k)\phi(t-k)$, we have $\widetilde{Y}^{(n)}(t) = \widetilde{Y}(t)$ for all $t\in (T,n+1)$ (cf.\ Moulines et al.\ \cite{moulines:roueff:taqqu:2007:JTSA}). Noting $D(2^j,k) = \int_\bbR \widetilde Y(t) 2^{-j/2}\psi(2^{-j}t - k)dt$ and $D^{(n)}(2^j,k) = \int_\bbR \widetilde Y^{(n)}(t)2^{-j/2}\psi(2^{-j}t - k)dt$, it follows that the finite-sample wavelet coefficients $D^{(n)}(2^j,k)$ of $\widetilde{Y}^{(n)}(t) $ are equal to $D(2^j,k)$ whenever $\textnormal{supp } \psi(2^{-j} \cdot - k) = (2^jk, 2^j (k+T)) \subseteq (T,n+1)$. In other words,
\begin{equation}\label{e:d^n=d}
D^{(n)}(2^j,k) = D(2^j,k),\qquad \forall (j,k)\in\{(j,k): 2^{-j}T\leq k\leq 2^{-j}(n+1)-T\}.
\end{equation}
Equivalently, such subset of finite-sample wavelet coefficients is not affected by the so-named \textit{border effect} (cf.\ Craigmile et al.\ \cite{craigmile:guttorp:Percival:2005}, Percival and Walden \cite{percival:walden:2006}, Didier and Pipiras \cite{didier:pipiras:2010}). Moreover, by \eqref{e:d^n=d} the number of such coefficients at octave $j$ is given by $n_j = \lfloor 2^{-j}(n+1-T)-T\rfloor$. Hence, $n_j \sim 2^{-j}n$ for large $n$. Thus, for notational simplicity we suppose
 \begin{equation}\label{e:nj=n/2^j}
 n_{j} = \frac{n}{2^j}
 \end{equation}
 holds exactly and only work with wavelet coefficients unaffected by the border effect.

\subsection{Wavelet random matrices and assumptions} \label{s:wavelet-assumptions}

Throughout the paper, we assume observations stem from the model \eqref{e:Y(t)}. The independent ``signal" $X = \{X(t)\}$ and the noise component $Z = \{Z(t)\}$ are $\bbR^r$-valued and $\bbR^p$-valued, respectively, where $r$ is fixed and $p = p(n)$.  The deterministic matrix $P = P(n)$ can be expressed as
\begin{equation}\label{e:P(n)}
M(p,r,\bbR) \ni P(n) = \Big( \mathbf{p}_{ 1}(n),\hdots,\mathbf{p}_{ r}(n) \Big), \quad \|\mathbf{p}_{ q}(n)\| = 1, \hspace{3mm} q = 1,\hdots,r.
\end{equation}
For the sake of clarity and mathematical generality, we state the appropriate conditions for the convergence in probability as well as for the asymptotic normality of wavelet log-eigenvalues directly in the wavelet domain. For $j \in \bbN$, $k \in \bbZ$ and a dyadic sequence $\{a(n)\}_{n \in \bbN}$, the random vectors
$$
D(a(n)2^j,k)  \in \bbR^p, \quad  D_X(a(n)2^j,k)   \in \bbR^r \quad \textnormal{and} \quad D_Z(a(n)2^j,k) \in \bbR^p
$$
denote the wavelet transform at scale $a(n)2^j$ of the stochastic processes $Y$, $X$ or $Z$, respectively. Whenever well-defined, the \textit{wavelet random matrix} -- or \textit{sample wavelet (co)variance} -- of $Y$ at scale $a(n)2^j$ is denoted by
\begin{equation}\label{e:W(a(nu))}
{\mathcal S}_{\geq 0}(p,\bbR) \ni \mathbf{W}(a(n) 2^j)) = \frac{1}{n_{a,j}}\sum^{n_{a,j}}_{k=1}D(a(n)2^j,k)D(a(n)2^j,k)^*, \quad n_{a,j} = \frac{n}{a(n)2^j}.
\end{equation}
The remaining wavelet random matrix terms $\W_X,\W_{XZ},\W_Z$ are naturally defined as
$$
\mathbf{W}_X(a(n) 2^j)) = \frac{1}{n_{a,j}}\sum^{n_{a,j}}_{k=1}D_X(a(n)2^j,k)D_X(a(n)2^j,k)^* \in {\mathcal S}_{\geq 0}(r,\bbR),
$$
$$
\mathbf{W}_{X,Z}(a(n) 2^j)) = \frac{1}{n_{a,j}}\sum^{n_{a,j}}_{k=1}D_X(a(n)2^j,k)D_Z(a(n)2^j,k)^* \in M(r,p,\bbR),
$$
\begin{equation}\label{e:WZ(a(n)2^j)}
\mathbf{W}_{Z}(a(n) 2^j)) = \frac{1}{n_{a,j}}\sum^{n_{a,j}}_{k=1}D_Z(a(n)2^j,k)D_Z(a(n)2^j,k)^* \in {\mathcal S}_{\geq 0}(p,\bbR).
\end{equation}
Further define the auxiliary random matrix
\begin{equation}\label{e:B-hat_a(2^j)}
\widehat{{\mathbf B}}_a(2^j) = P_H^{-1}a(n)^{-H-(1/2)I}{\mathbf W}_X(a(n)2^j)a(n)^{-H^*-(1/2)I} (P_H^*)^{-1}\in {\mathcal S}_{\geq 0}(r,\bbR),
\end{equation}
as well as its mean ${\mathbf B}_a(2^j) := \bbE \widehat{{\mathbf B}}_a(2^j)$. The matrix $\widehat{{\mathbf B}}_a(2^j)$ should be interpreted as a version of ${\mathbf W}_X(a(n)2^j)$ after compensating for scaling and non-canonical coordinates (cf.\ relation \eqref{e:B(2^j)_entrywise_scaling}, which displays canonical scaling). In \eqref{e:B-hat_a(2^j)}, we assume that the scaling matrix $H$ has the Jordan form
\begin{equation}\label{e:H_is_diagonalizable}
H = P_H \textnormal{diag}(h_1,\hdots,h_r) P^{-1}_H, \quad  P_H \in GL(r,\bbR), \quad -1/2 < h_1 \leq \hdots \leq h_r <\infty.
\end{equation}
We make use of the following assumptions in the main results of this paper (Section \ref{s:main}). For expository purposes, we first state the assumptions, and then provide some interpretation.

\medskip

\noindent {\sc Assumption $(A1)$}: The wavelet random matrix
$$
{\mathbf W}(a(n)2^j) = P(n){\mathbf W}_X(a(n)2^j)P^*(n) + {\mathbf W}_Z(a(n)2^j)
$$
\begin{equation}\label{e:W=PWXP*+WZ+PWXZ+WXZ*P*}
 + P(n){\mathbf W}_{X,Z}(a(n)2^j)+ {\mathbf W}^*_{X,Z}(a(n)2^j)P^*(n)
\end{equation}
is well defined a.s.

\medskip

\noindent {\sc Assumption $(A2)$}: In \eqref{e:W=PWXP*+WZ+PWXZ+WXZ*P*}, for $H$ as in \eqref{e:B-hat_a(2^j)},
\begin{equation}\label{e:assumptions_WZ=OP(1)}
\max \Big\{\|\E \mathbf{W}_Z(a(n)2^j)\|, \|\W_Z(a(n)2^j)\| , \|a(n)^{-H-(1/2)I} \mathbf{W}_{X,Z}(a(n) 2^j))\|\Big\} = O_{\bbP}(1).
\end{equation}

\medskip

\noindent {\sc Assumption $(A3)$}: The random matrix $\widehat{{\mathbf B}}_a(2^j)$ as in \eqref{e:B-hat_a(2^j)} satisfies
\begin{equation}\label{e:sqrt(K)(B^-B)->N(0,sigma^2)}
\Big(\hspace{1mm}\sqrt{n_{a,j}}\hspace{0.5mm}(\vecoper_{{\mathcal S}}\widehat{{\mathbf B}}_a(2^j) - \vecoper_{{\mathcal S}}{\mathbf B}_a(2^j))\hspace{1mm}\Big)_{j=j_1,\hdots,j_m}\stackrel{d}\rightarrow {\mathcal N}(0,\Sigma_B(j_1,\hdots,j_m)), \quad n \rightarrow \infty.
\end{equation}
for some $\Sigma_B(j_1,\ldots,j_m)\in \mathcal{S}_{\geq 0}\big(m \cdot \frac{r(r+1)}{2},\bbR\big)$. In addition, for some $\beta>0$,
\begin{equation}\label{e:|bfB_a(2^j)-B(2^j)|=O(shrinking)}
\|{{\mathbf B}}_a(2^j) - \mathbf B(2^j) \| = O(a(n)^{-\beta}), \quad  n\to\infty, \quad j=j_1,\hdots,j_m,
\end{equation}
where
\begin{equation}\label{e:B(2^j)_full_rank}
\mathbf B(2^j) \in \mathcal S_{>0}(r,\bbR). %
\end{equation}

\medskip

\noindent {\sc Assumption $(A4)$}: The dimension $p(n)$ and the scaling factor $a(n)$ satisfy the relations
\begin{equation}\label{e:p(n),a(n)_conditions}
a(n) \leq \frac{n}{2^{j_2}}, \quad \frac{a(n)}{n}+ \frac{n}{a(n)^{2 \varpi +1}} \rightarrow 0, \quad \frac{p(n)}{n/a(n)} \rightarrow c \in [0,\infty), \quad n \rightarrow \infty,
\end{equation}
where $\varpi>0$ is defined as
\begin{equation}\label{e:varpi_parameter}
\varpi =
\min\Big\{\min_{\{2\leq i \leq r\!~:\!~h_i-h_{i-1}>0\}} (h_i-h_{i-1}),~ \beta,~\frac{h_1}{2} + \frac{1}{4} \Big\}.
\end{equation}

\medskip

\noindent {\sc Assumption $(A5)$}: Let $P(n) \in M(p,r,\bbR)$ and $P_H \in GL(r,\bbR)$ be as in \eqref{e:P(n)} and \eqref{e:H_is_diagonalizable}, respectively. Let $P(n)P_H=Q(n)R(n)$ be the $QR$ decomposition of $P(n)P_H$, where $R(n)\in GL(r,\bbR)$ and $Q(n) \in M(p,r,\bbR)$ has orthonormal columns. Let $\varpi$ be as in \eqref{e:varpi_parameter}. Then, there exists a (deterministic) matrix $A \in \mathcal{S}_{>0}(r,\bbR)$ with Cholesky decomposition
$A = R^* R$
such that
\begin{equation}\label{e:<p1,p2>=c12_2}
\|R(n)-R\| = O(a(n)^{-\varpi}).
\end{equation}\vspace{1ex}

Assumptions $(A1 - A3)$ are stated in the wavelet domain. Assumption $(A1)$ holds under very general conditions. In fact, since $X$ and $Z$ are assumed independent, it suffices that $\|{\mathbf W}_{X}(a(n)2^j)\| < \infty$ and $\|{\mathbf W}_{Z}(a(n)2^j)\| < \infty$ a.s., which holds as long as $X$ and $Z$ are well-defined discrete time stochastic processes. Assumption $(A2)$ ensures that the  influence of the random matrices ${\mathbf W}_{X,Z}(a(n)2^j)$ and ${\mathbf W}_{Z}(a(n)2^j)$ in the observed wavelet spectrum ${\mathbf W}(a(n)2^j)$ is not too large. Assumption $(A3)$ posits the asymptotic normality of the (wavelet domain) fractional component ${\mathbf W}_{X}(a(n)2^j)$ after compensating for scaling and non-canonical coordinates. %

In turn, Assumption $(A4)$ controls the divergence rates among $n$, $a(n)$ and $p(n)$. In particular, it states that the scaling factor $a(n)$ must blow up slower than $n$, and that the three-component ratio $\frac{p(n) a(n)}{n}$ must converge to a constant (cf.\ the traditional ratio $\lim_{n \rightarrow \infty} p(n)/n$ for high-dimensional sample covariance matrices). Assumption $(A5)$ ensures that, asymptotically speaking, the angles between the column vectors of the matrix $P(n) P_H$ converge in such a way that the matrix of asymptotic angles $\lim_{n \rightarrow \infty}P^*_H P^*(n)P(n)P_H = A$ has full rank. This entails that $P(n)$ does not strongly perturb the scaling properties of the hidden random matrix ${\mathbf W}_{X}(a(n)2^j)$.

In the following examples, to fix ideas we briefly illustrate contexts where assumptions $(A2)$ and $(A3)$ hold.

\begin{example}\label{ex:Prop_4_1_in_Abry_et_al}
Suppose that, for any fixed $n$, $X$ is an $r$-variate ofBm and the $p(n)$-variate process $Z$ are independent. Further assume that, for fixed integers ${\mathfrak p}$ and ${\mathfrak q}$, $Z$ is made up of entry-wise independent, i.d.\ ARMA$({\mathfrak p},{\mathfrak q})$ processes. Then, under mild regularity assumptions it can be shown that condition \eqref{e:assumptions_WZ=OP(1)} (i.e., $(A2)$) is satisfied. It can further be proved that relations \eqref{e:sqrt(K)(B^-B)->N(0,sigma^2)} and \eqref{e:|bfB_a(2^j)-B(2^j)|=O(shrinking)} hold for the matrix sequences $\widehat {\mathbf B}_a(2^j) $, ${\mathbf B}_a(2^j) \in {\mathcal S}_{\geq0}(\bbR,r)$ and a matrix ${\mathbf B}(2^j) \in {\mathcal S}_{>0}(\bbR,r)$ (i.e., $(A3)$) (see Example 2.1 and Proposition 4.1 in Abry et al.\ \cite{abry:boniece:didier:wendt:2022:wavelet_eigenanalysis}; cf.~Lemma C.2 in Abry and Didier \cite{abry:didier:2018:dim2}). Moreover,
when $H$ is diagonalizable with real eigenvalues, it can be easily shown that the matrix $\B(2^j)$ satisfies the entrywise (i.e., along canonical axes) scaling relations
\begin{equation}\label{e:B(2^j)_entrywise_scaling}
{\mathbf B}(2^j) = \Big( {\mathbf b}(2^j)_{ii'} \Big)_{i,i'=1,\hdots,r} = \Big( 2^{j(h_i + h_{i'}+1)}{\mathbf b}(1)_{ii'} \Big)_{i,i'=1,\hdots,r}.
\end{equation}
\end{example}

\begin{example}\label{ex:non-Gaussian_linear}
Let
$$
X = \{X(t)\}_{t \in \bbZ} = \big\{(X_1(t),\hdots,X_r(t))^* \big\}_{t \in \bbZ}
$$
be a possibly non-Gaussian, $\bbR^r$-valued stochastic process whose entry-wise components are independent linear fractional processes with finite fourth moments. Then, based on the framework constructed in Roueff and Taqqu \cite{roueff:taqqu:2009}, one can show that the associated random matrix $\widehat{{\mathbf B}}_a(2^j) \in {\mathcal S}_{>0}(r,\bbR)$ satisfies conditions \eqref{e:sqrt(K)(B^-B)->N(0,sigma^2)} and \eqref{e:|bfB_a(2^j)-B(2^j)|=O(shrinking)} (i.e., $(A3)$) under conditions $(W1-W4)$ and mild additional assumptions on the wavelet $\psi$ and on the process $X$ (see Proposition C.2 in Abry et al.\ \cite{abry:boniece:didier:wendt:2022:wavelet_eigenanalysis}). For this instance of $X$, the model \eqref{e:Y(t)} is associated with the so-named blind source separation problems in the field of signal processing (e.g., Comon and Jutten \cite{comon:jutten:2010}; see also Abry et al.\ \cite{abry:didier:li:2019} on fractional instances).
\end{example}

\begin{example}\label{ex:Haar_sub-Gaussian}
Recall that a distribution is called \textit{sub-Gaussian} when its tails are no heavier than those of the Gaussian distribution (Vershynin \cite{vershynin:2018}, Proposition 2.5.2). Sub-Gaussian distributions form a broad family that includes the Gaussian distribution itself, as well as compactly supported distributions, for example. Suppose the noise process $\{Z(t)\}_{t \in \bbZ}$ consists of i.i.d.\ sub-Gaussian observations. Consider the Haar wavelet framework, where the wavelet coefficients are computed by means of Mallat's iterative procedure \eqref{e:Mallat}. Then, it is possible to show that the wavelet random matrix ${\mathbf W}_{Z}(a(n)2^j)$ satisfies condition \eqref{e:assumptions_WZ=OP(1)} (i.e., $(A2)$; see the discussion in Abry et al.\ \cite{abry:boniece:didier:wendt:2022:wavelet_eigenanalysis}, Example 4.3). On possible extensions beyond the sub-Gaussian case as well as on related results, see Vershynin \cite{vershynin:2012} and  Einmahl and Li \cite{einmahl:li:2008}.
\end{example}

\section{Main results}\label{s:main}

As mentioned in the Introduction, under the assumptions laid out in Section \ref{s:wavelet-assumptions} the $r$ largest eigenvalues of the random matrix $\W(a(n)2^j)$ display \emph{asymptotic} scaling relationships in high dimensions dictated by the eigenvalues of $H$ (see \eqref{e:lambda_p-r+q_scales}). Furthermore, under additional assumptions, the fluctuations of the rescaled wavelet log-eigenvalues around their limits are asymptotically Gaussian. For the reader's convenience, the precise statements are provided in Theorems \ref{t:lim_n_a*lambda/a^(2h+1)} and \ref{t:asympt_normality_lambdap-r+q}. %
These facts indicate that the scaling eigenvalues \eqref{e:H=PHdiag(h1,...,hn)P^(-1)H} can be efficiently estimated in high dimensions through a linear regression based on the log-eigenvalues of $\W(a(n)2^j)$ over a range of scales $2^{j_1},\ldots,2^{j_2}$, provided the scaling factor $a(n)$ is sufficiently large. This motivates the following definition.

\begin{definition}\label{def:eigenvalue_estimator}
Let $\{{\mathbf W}(a(n)2^j)\}_{j=j_1,\hdots,j_2}$ be the wavelet random matrices corresponding to scales $\{a(n)2^{j_1},\hdots, a(n)2^{j_2}\}$. Fix a range of octaves
\begin{equation}\label{e:j1<...<j2_m=j2-j1+1}
j = j_1 , j_1 + 1, \hdots , j_2, \quad m := j_2 - j_1 + 1.
\end{equation}
The \textit{wavelet eigenvalue regression} is given by
\begin{equation}\label{e:hl-hat}
\{ \widehat{\ell}_{i} \}_{i = 1,\hdots,p(n)} :=  \Big\{\frac{1}{2}\Big(\sum_{j=j_1}^{j_2} w_j \log_2 \lambda_{i}\big({\mathbf W}(a(n)2^j)\big) -1\Big)\Big\}_{i = 1,\hdots,p(n)}.
\end{equation}
 In \eqref{e:hl-hat}, $w_j$, $j = j_1,\hdots,j_2$, are weights satisfying the relations
\begin{equation}\label{e:sum_wj=0,sum_jwj=1}
\sum^{j_2}_{j=j_1}w_j = 0, \quad \sum^{j_2}_{j=j_1}j w_j = 1,
\end{equation}
where $w_j = 1$ if $m = 1$ (see \eqref{e:j1<...<j2_m=j2-j1+1}).
\end{definition}
If ${\mathbf W}(a(n)2^j) \in {\mathcal S}_{> 0}(r,\bbR)$, then the estimator \eqref{e:hl-hat} is well defined. Assuming for the moment that the dimension $r$ of the hidden fractional process $X$ is known, expression \eqref{e:hl-hat} based on the top $r$ wavelet eigenvalues can be interpreted as wavelet eigenstructure estimators of the scaling eigenvalues, i.e., we can write
\begin{equation}\label{e:hq=ell_p-r+q}
\{ \widehat{h}_q \}_{q = 1,\hdots,r}=\{ \widehat{\ell}_{p(n)-r+q} \}_{q = 1,\hdots,r}.
\end{equation}
Hence, the estimator \eqref{e:hq=ell_p-r+q} can be viewed as a high-dimensional extension of the wavelet eigenvalue regression estimator first developed in Abry and Didier \cite{abry:didier:2018:dim2,abry:didier:2018:n-variate} in a noiseless, low-dimensional context for ofBm.

In the following theorems, we characterize the asymptotic and large scale behavior of the wavelet eigenvalue regression estimator \eqref{e:hq=ell_p-r+q} in high dimensions, namely, its consistency and, under assumptions, asymptotic normality. Consistency is, indeed, a direct consequence of the convergence in probability of the $r$ largest rescaled wavelet eigenvalues (cf.\ Theorem \ref{t:lim_n_a*lambda/a^(2h+1)}). By contrast, establishing the asymptotic normality of the wavelet eigenvalue regression estimator goes beyond the asymptotic distribution of wavelet eigenvalues over large scales in high dimensions (cf.\ Theorem \ref{t:asympt_normality_lambdap-r+q}). In fact, it also requires establishing the rate of convergence of the rescaled $r$ largest eigenvalues of the deterministic wavelet matrices $\bbE \W (a(n)2^j)$. In turn, this involves studying the rate of angular convergence of the associated (deterministic) wavelet eigenvectors in terms of the high-dimensional coordinates $P = P(n)$.
\begin{theorem}\label{t:h-hat_consistency}
Assume $(W1-W4)$ and $(A1-A5)$ hold and fix $j_1,j_2 \in \bbN$, $j_1<j_2$. Then,
\begin{equation}
\Big(\widehat{h}_q - h_q \Big)_{q=1,\hdots,r} \stackrel{\bbP}\rightarrow 0, \quad n \rightarrow \infty.
\end{equation}
\end{theorem}
\begin{theorem}\label{t:h-hat_asympt_normality} Suppose $(W1-W4)$ and $(A1-A5)$ hold, and fix $j_1,j_2 \in \bbN$, $j_1<j_2$. Further suppose one of the following conditions holds, namely,
\begin{itemize}
\item[(i)] either
\begin{equation}\label{e:assumption:simple_Hurst}
-1/2 < h_1 < \ldots < h_r < \infty;
\end{equation}
\item[(ii)] or $h_1 = \ldots = h_r$ and the functions $\xi_q(2^j)$ in \eqref{e:lim_n_a*lambda/a^(2h+1)} satisfy
\begin{equation}\label{e:xiq_distinct}
q_1\neq q_2 \Rightarrow \xi_{q_1}(1)\neq \xi_{q_2}(1).
\end{equation}
\end{itemize}
Then,
\begin{equation}\label{e:h^q_asymptotically_normal}
\sqrt{\frac{n}{a(n)}} \Big( \widehat{h}_q - h_q \Big)_{q=1,\hdots,r} \stackrel{d}\rightarrow {\mathcal N}(0,M \Sigma_{\lambda} M^*),
\end{equation}
as $n \rightarrow \infty$, for some weight matrix $M$ (see \eqref{e:weight_matrix_M}) and $\Sigma_{\lambda}$ as in Theorem \ref{t:asympt_normality_lambdap-r+q}.
\end{theorem}

\begin{remark}\label{r:asympt_rescaled_eigenvalues}
Under non-simple scaling matrix eigenvalues and without condition \eqref{e:xiq_distinct}, the asymptotic distribution of $\{\widehat{h}_q\}_{q=1,\hdots,r}$ is generally expected to be non-Gaussian because, in this case, wavelet log-eigenvalues themselves are generally expected to be non-Gaussian (see the discussion in Abry et al.~(2022), Remark 3.1). %
The broad characterization of the distribution of wavelet log-eigenvalues outside the framework of the assumptions of Theorem \ref{t:h-hat_asympt_normality} is currently a topic of research.
\end{remark}
In practice, the dimension $r$ of the latent process $X$ may not be known.  In the following definition, we introduce an estimator of $r$ based on the eigenvalues of  $\W(a(n)2^j)$.
\begin{definition}\label{def:r-hat}
Let $j_1< j_2$ and consider any set of weights $v_{j_1},\ldots v_{j_2}\in \bbR$ satisfying $\sum_{j=j_1}^{j_2}v_j=1$. Let
\begin{equation}\label{e:def_deltaq}
\Delta_i(j_1,j_2):= \sum_{j=j_1}^{j_2}v_j\frac{\log\lambda_i\big(\mathbf W(a(n)2^j)\big)}{\log(a(n)2^j)}
\end{equation}
Given $\kappa>0$, we define
\begin{equation}\label{e:def_rhat}
{\widehat{r}}(2^{j_1},2^{j_2},\kappa) := \#\{i:\Delta_i(j_1,j_2)>\kappa\}.
\end{equation}
\end{definition}
Now note that, under the assumptions of Theorem \ref{t:h-hat_consistency}, the lowest $p-r$ eigenvalues stay bounded. Hence, for $i\leq p-r$, the quantity $\Delta_i(j_1,j_2)$ tends to zero in probability. On the other hand, for $i>p-r$, still under the assumptions of Theorem \ref{t:h-hat_consistency}, $\Delta_i(j_1,j_2)$ converges to $2 h_{i-(p-r)}+1>0$ in probability, thereby separating the non-scaling and scaling eigenvalues of ${\mathbf W}(a(n)2^j)$. This phenomenon lies behind the following theorem, which establishes the consistency of the estimator $\widehat{r}(2^{j_1},2^{j_2},\kappa)$. %
\begin{theorem}\label{t:r-hat->r}
Let $\widehat{r}(2^{j_1},2^{j_2},\kappa)$ be as in \eqref{e:def_rhat} with $\kappa \in (0,2h_1+1]$ and suppose the assumptions of Theorem \ref{t:h-hat_consistency} hold. Then,
\begin{equation}\label{e:r-hat-->r}
\lim_{n\to\infty}  \bbP\big(\widehat r(2^{j_1},2^{j_2},\kappa) = r\big) =1%
\end{equation}
\end{theorem}
\begin{remark} In the context of the (univariate) wavelet regression, it is common practice to select the weights \eqref{e:sum_wj=0,sum_jwj=1} to include information about the wavelet variance over the regression scales $2^{j_1},\ldots 2^{j_2}$ (cf.\ Veitch and Abry \cite{veitch:abry:1999}). This can lead to improved estimation performance. For convenience, in all simulations we set $v_j=j w_j$, where $w_j$ are the regression weights in \eqref{e:hl-hat}.
\end{remark}

\section{Monte Carlo studies}\label{s:MC}

In this section, we describe our Monte Carlo-based studies of the asymptotic behavior of the wavelet eigenvalue regression estimator $\{\widehat{\ell}_i\}_{i=1,\hdots,p(n)}$ (see \eqref{e:hl-hat}) and of the estimator $\widehat{r}(j_1,j_2,\kappa)$ (see \eqref{e:def_rhat}) in high dimensions.

Figure \ref{fig:ellhat} displays the asymptotic behavior of $\{\widehat{\ell}_i\}_{i=1,\hdots,p(n)}$ over scales $[j-1,j+1]$ in the specified ranges of $j$. As the effective sample size $n_j=n/2^j$ and dimension $p = p(n)$ increase together at the specified ratio $c := p/n_j$,  the estimators display convergence to their theoretical values or boundedness.  As expected, at more favorable (namely, smaller) values of $c$, the eigenvalues display near-convergence to theoretical values at smaller sample sizes. Computational studies not shown reveal that the asymptotic performance of $\widehat \ell_i$ over multiple scales in Figure \ref{fig:ellhat} displays reduced bias by comparison to using a single or a reduced number of scales.

Recall that the so-named Gamma plots are defined as plots of the empirical quantiles of squared Mahalanobis distance vs.\ the theoretical quantiles of a $\chi^2_r$ distribution (Johnson and Wichern \cite{johnson:wichern:2002}, Oppong and Agbedra \cite{oppong:agbedra:2016}). Figure \ref{fig:asympt_h-hat_qqplot} displays Gamma plots of the sample distribution of the wavelet eigenvalue regression estimator $\{\widehat{h}_q\}_{q=1,\hdots,6}= \{\widehat{\ell}_{p(n)-r + q}\}_{q=1,\hdots,6}$ ($r = 6$) for various combinations of $p$ and $n_{j_2} = n/2^{j_2}$. In all instances, the hidden process $X$ was simulated by means of CME as an ofBm with scaling (Hurst) parameter $H = \diag\big(0.1,0.3,0.5,0.6,0.8,0.9\big)$ and (instantaneous) covariance matrix $\bbE B_H(1)B_H(1) = \mbox{Toeplitz}(1, 0.2, 0.2, 0.3, 0.2, 0.3)$. At each time $t$, the vector noise term $Z(t)$ consisted of $p$ i.i.d.\ ${\mathcal N}(0,1)$ entries. At each run, the columns of the coordinates matrix $P(n)$ were set to the first $r = 6$ canonical vectors. As $n$ (and, hence, $p$) grows, in the high-dimensional limit the distribution of $\{\widehat{h}_q\}_{q=1,\hdots,6}$ becomes increasingly closer to a multivariate Gaussian, as mathematically characterized in Theorem \ref{t:h-hat_asympt_normality}.

Figure \ref{fig:r-hat} depicts the performance $\widehat{r}(j_1,j_2,\kappa)$ as a function of the threshold $\kappa > 0$ for various combinations of $p/n_{j_2}$, $n_{j_2}$ and $r$. To illustrate the asymptotic behavior of $\widehat{r}(j_1,j_2,\kappa)$ while keeping the ratio $c := p/n_{j_2}$ constant, the parameters $j_1,j_2$ were chosen as follows: for $n=2^{10}$, $(j_1,j_2)=(3,5)$ (hence, $n_{j_2}=32$); for $n=2^{12}$, $(j_1,j_2)=(4,6)$ (hence, $n_{j_2}=64$); for $n=2^{14}$, $(j_1,j_2)=(5,7)$ (hence, $n_{j_2}=128$). The hidden process $X$ was simulated by means of CME as an ofBm with scaling (Hurst) parameter $H = \diag\big(h_1,h_2,...,h_r\big)$, where $h_q = q/(r+1)$, $q = 1, \hdots,r$. At each time $t$, the vector noise term $Z(t)$ consisted of $p$ i.i.d.\ ${\mathcal N}(0,1)$ entries. At each run, the entries of the coordinates matrix $P(n)$ were generated as i.i.d.\ standard normal entries, and then the norm of each column was renormalized to 1 so as to keep constant the signal-to-noise ratio. For each wavelet random matrix, eigenvalues that fell below $10^{-10}$ in absolute value (stemming from deficient rank when $p/n_j < 1$) were set to 0 manually to avoid spuriously high values of $\Delta_i(j_1,j_2)$.

The computational studies reveal that, in all instances, too small or too large a value of $\kappa \in (0,1)$ leads to over-- and underestimation, respectively, of $X$ components. For any $r \in \{3,5,8\}$, and for each value of $c$, larger $n_{j_2}$ (and, hence, larger dimension $p$) produces a larger range of $\kappa$ where $\widehat{r}(j_1,j_2,\kappa)$ is perfectly accurate or nearly so, thus reflecting the high-dimensional asymptotics in Theorem \ref{t:r-hat->r}. Also, for each pair $n_{j_2}$ and $p$, larger values of $c$ lead to narrower ranges of $\kappa$ for which $\widehat r(j_1,j_2,\kappa)$ concentrates around the true value of $r$. This is expected since larger values of $c$ amount to more extreme stochastic regimes.

\section{Conclusion and open problems}\label{s:conclusion}

In this paper, we construct the wavelet eigenvalue regression methodology (Abry and Didier \cite{abry:didier:2018:dim2,abry:didier:2018:n-variate}) in high dimensions. We assume that possibly non-Gaussian, finite-variance $p$-variate measurements are made of a low-dimensional $r$-variate ($r \ll p$) fractional stochastic process with non-canonical scaling coordinates and in the presence of additive high-dimensional noise. The measurements are correlated both time-wise and between rows. Due to the asymptotic and large scale properties of large wavelet random matrices, the wavelet eigenvalue regression is shown to be consistent and (under additional assumptions) asymptotically Gaussian in the estimation of the fractal structure of the underlying measurements. We further construct a consistent estimator of the effective dimension $r$ of the system that significantly increases the robustness of the statistical methodology. The estimation performance over finite samples is further studied by means of simulations.

This research leads to many new research directions, some of which can be briefly described as follows. $(i)$ In applications, the results in this paper naturally pave the way for the investigation of scaling behavior in high-dimensional data from fields such as physics, neuroscience and signal processing; $(ii)$ Modeling requires a deeper study, in the wavelet domain, of the so-named \textit{eigenvalue repulsion effect} (e.g., Tao \cite{tao:2012}), which may severely skew the observed scaling laws when the assumptions of Theorem \ref{t:h-hat_asympt_normality} are violated. This is particularly important in the context of instances where all scaling eigenvalues are close to equal, with the same holding for the asymptotic rescaled eigenvalues $\xi_q(2^j)$ (see Wendt et al.\ \cite{wendt:abry:didier:2019:bootstrap} on preliminary computational studies). In those cases, it is of great interest to develop efficient testing procedures for the statistical identification of \textit{distinct }scaling eigenvalues in real data; $(iii)$ An interesting direction of extension is the construction of statistical methodology for instances where the $r$ largest eigenvalues of wavelet random matrices exhibit non-Gaussian fluctuations (cf.\ Remark \ref{r:asympt_rescaled_eigenvalues}), or other related instances where the hidden process $X$ displays heavier tails.

\begin{figure}
\begin{center}
\noindent\begin{tabular}{ccc}
$p/n_{j_2}=1/2$ &$p/n_{j_2}=1$ &$p/n_{j_2}=2$ \\
\\
\includegraphics[width=0.2\linewidth]{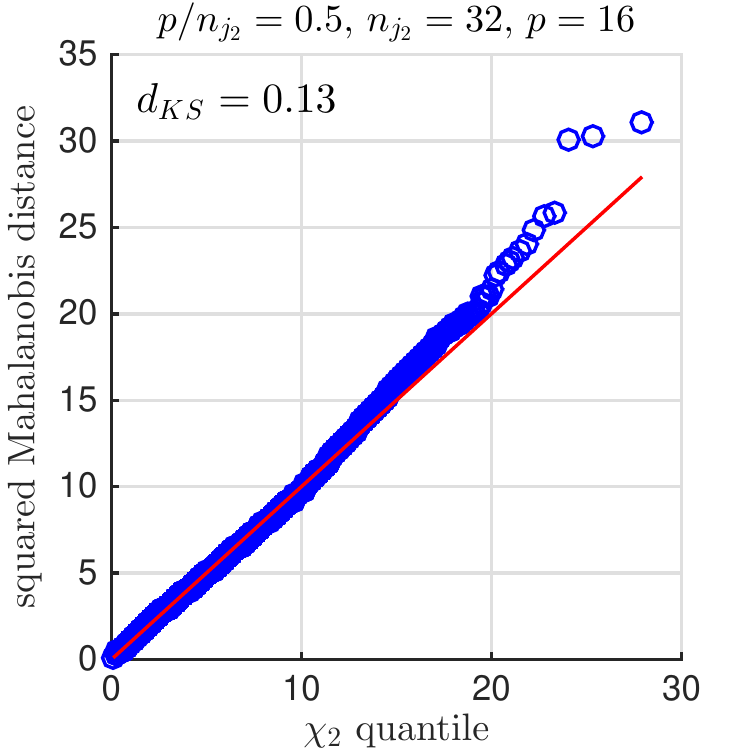}&\includegraphics[width=0.2\linewidth]{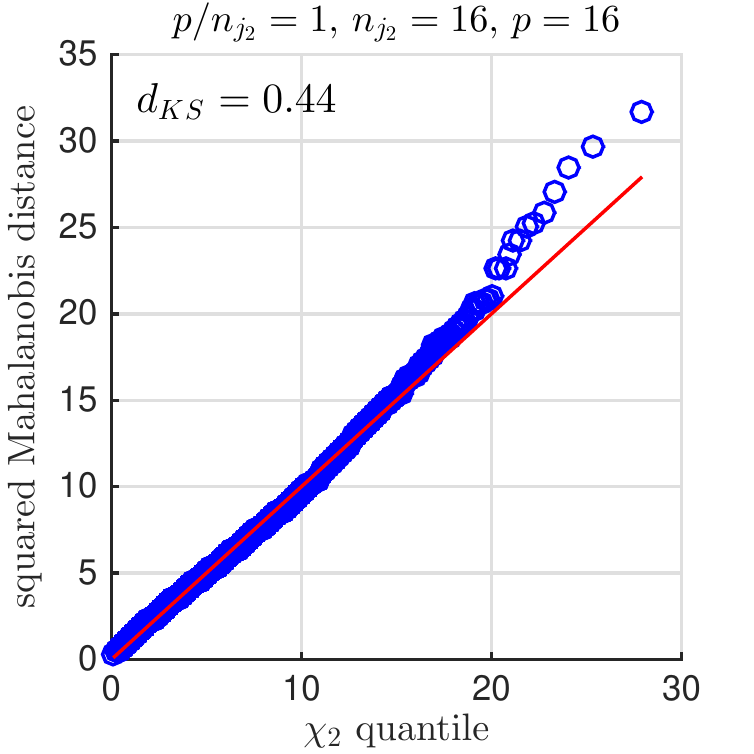}&\includegraphics[width=0.2\linewidth]{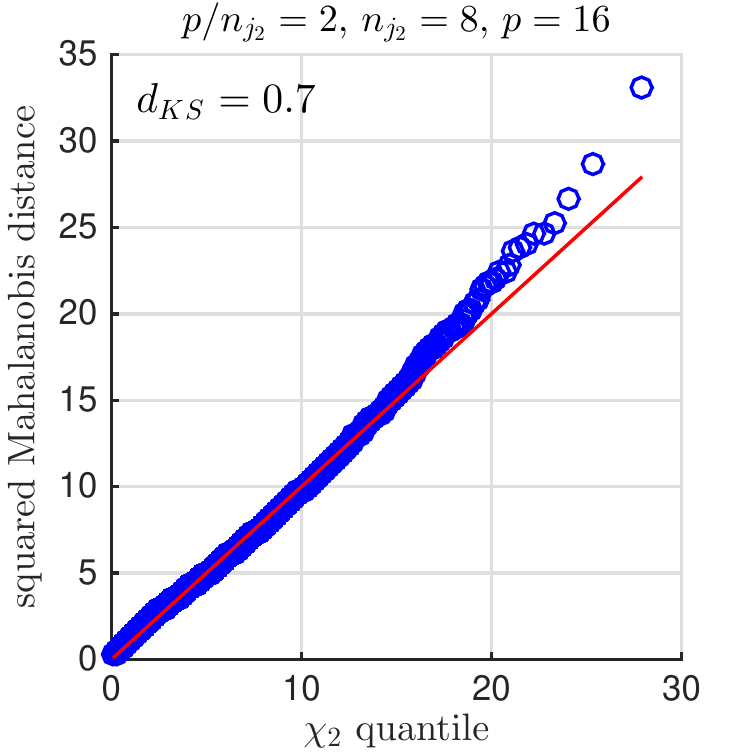}\\
\includegraphics[width=0.2\linewidth]{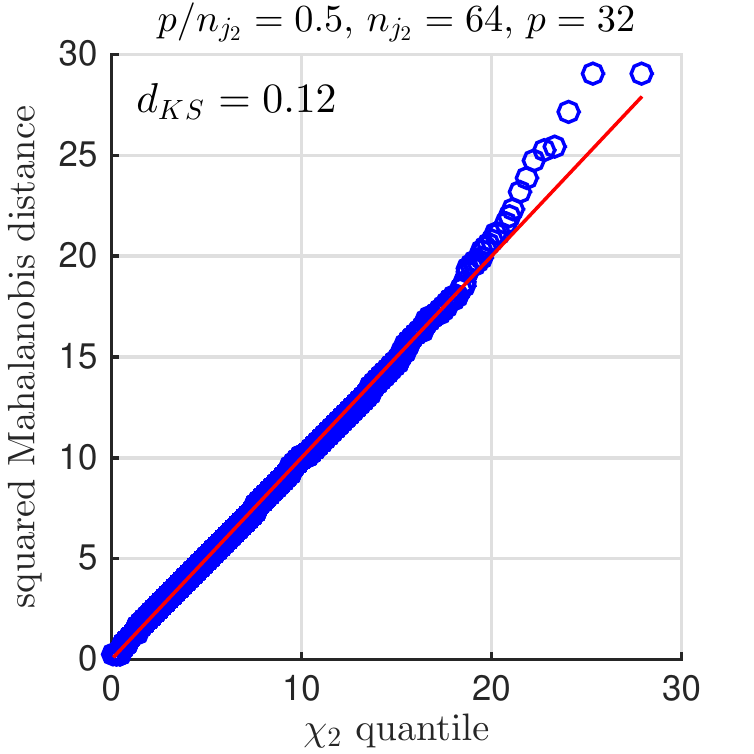}&\includegraphics[width=0.2\linewidth]{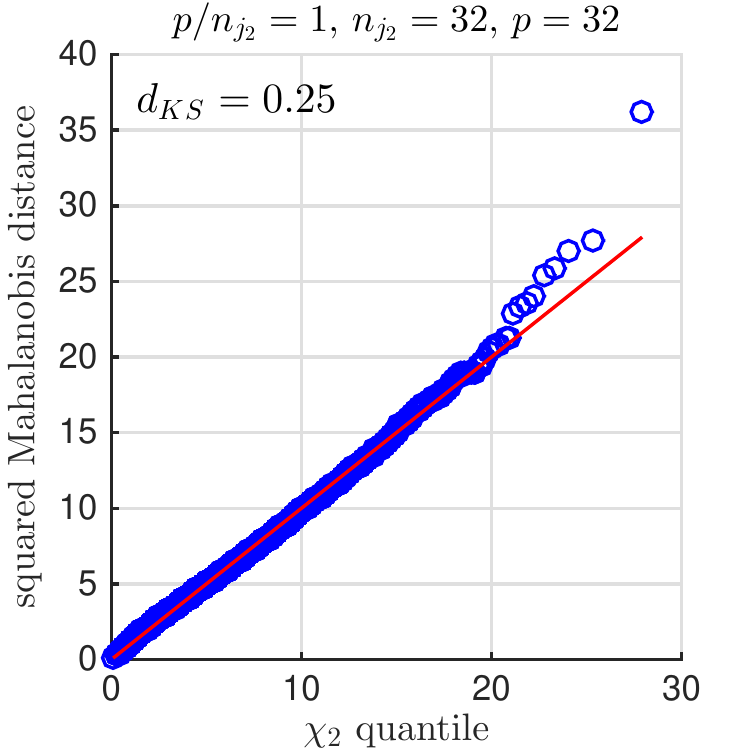}&\includegraphics[width=0.2\linewidth]{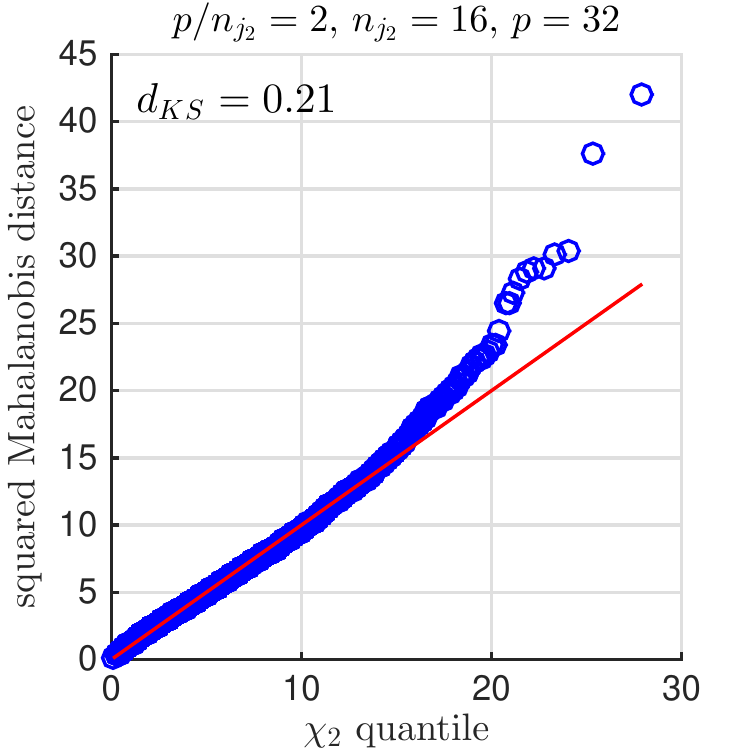}\\
\includegraphics[width=0.2\linewidth]{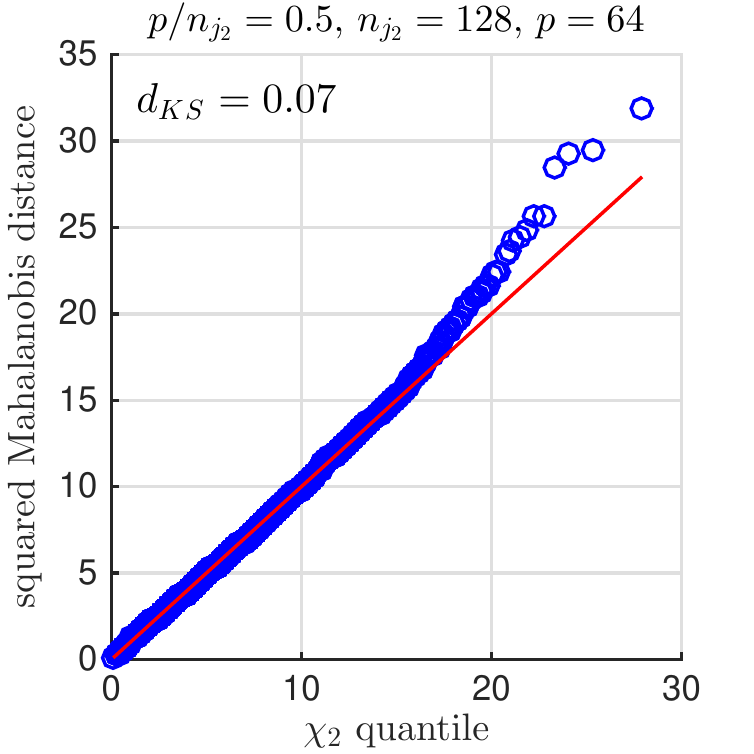}&\includegraphics[width=0.2\linewidth]{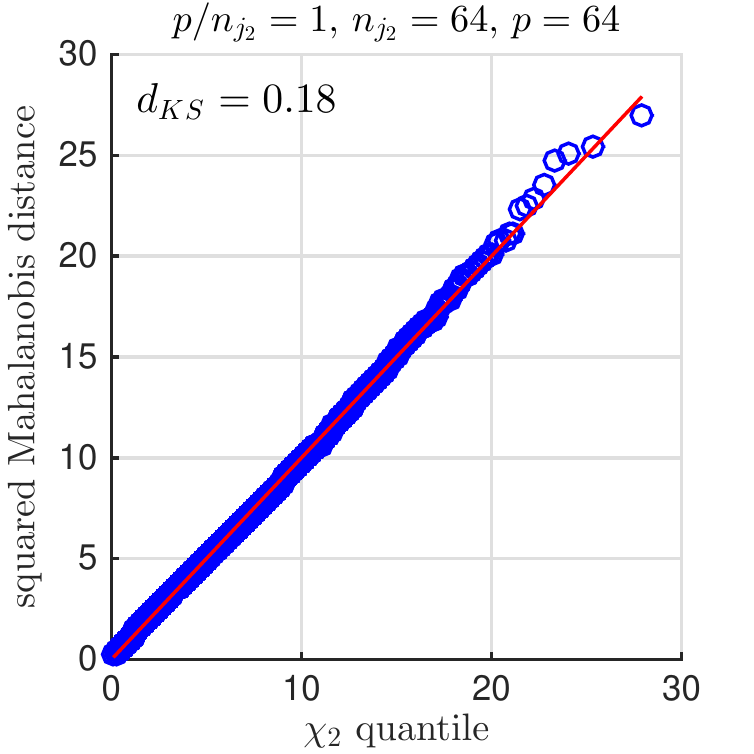}&\includegraphics[width=0.2\linewidth]{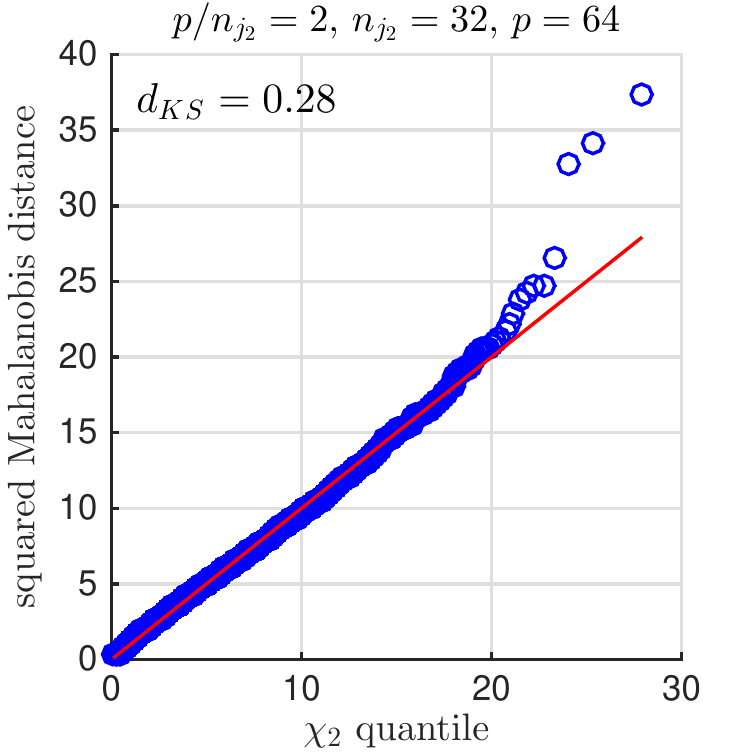}\\
\includegraphics[width=0.2\linewidth]{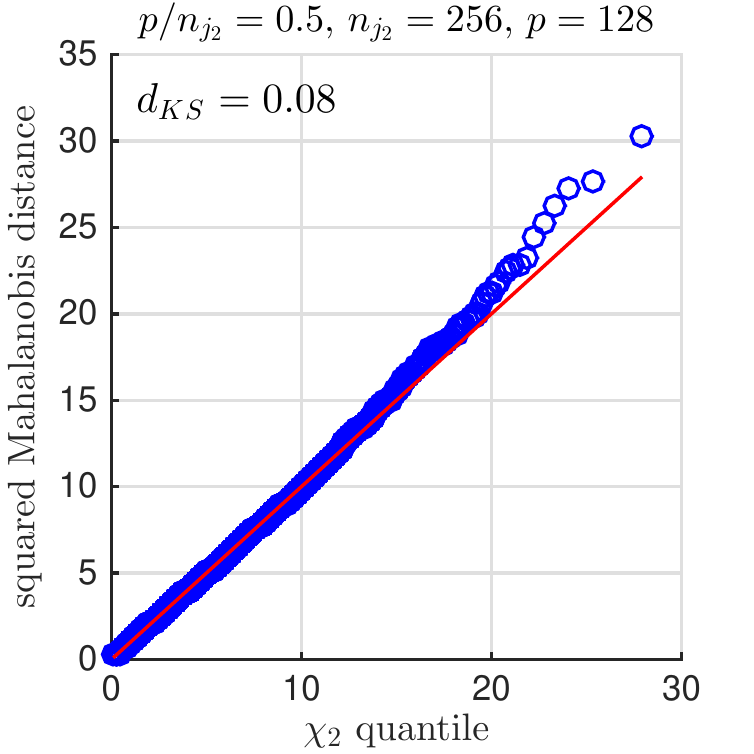}&\includegraphics[width=0.2\linewidth]{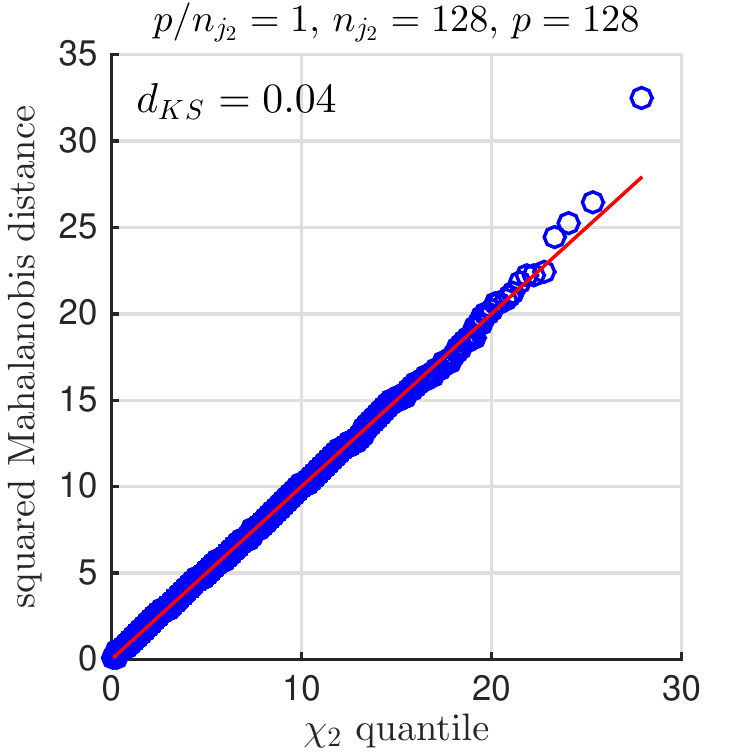}&\includegraphics[width=0.2\linewidth]{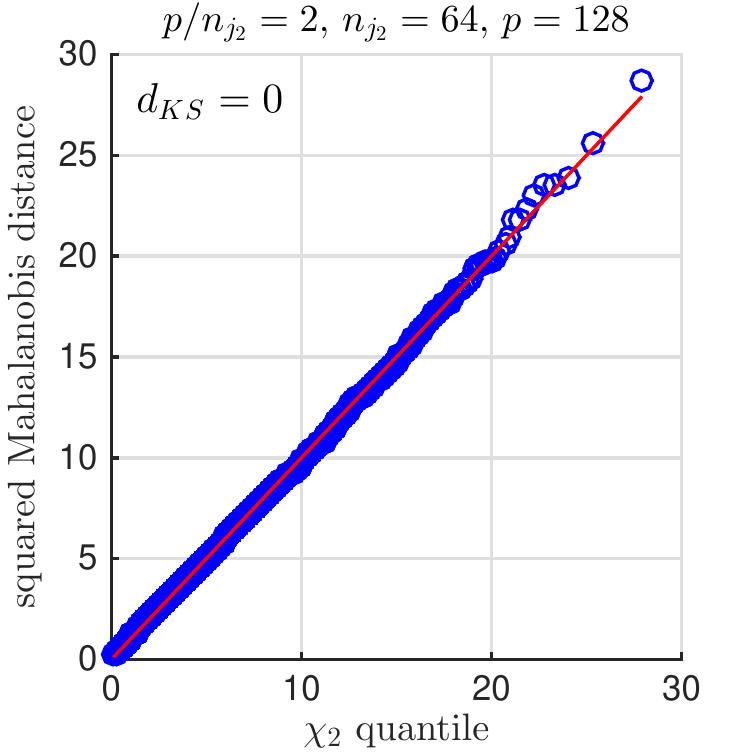}\\
\includegraphics[width=0.2\linewidth]{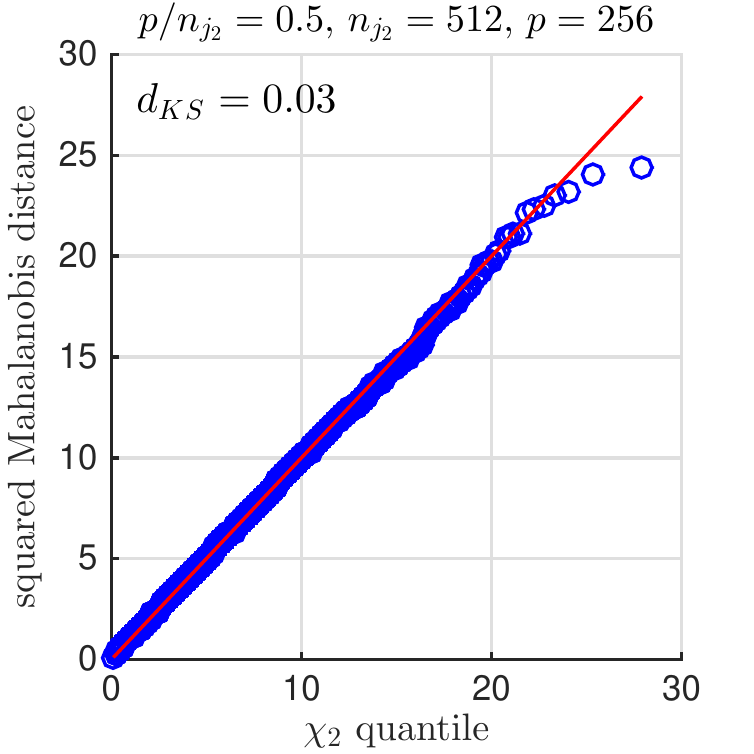}&\includegraphics[width=0.2\linewidth]{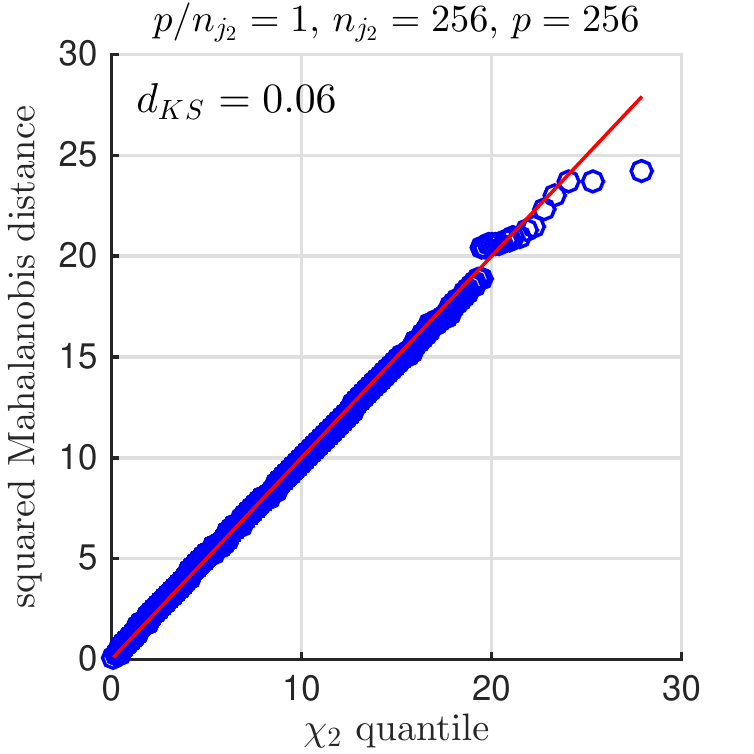}&\includegraphics[width=0.2\linewidth]{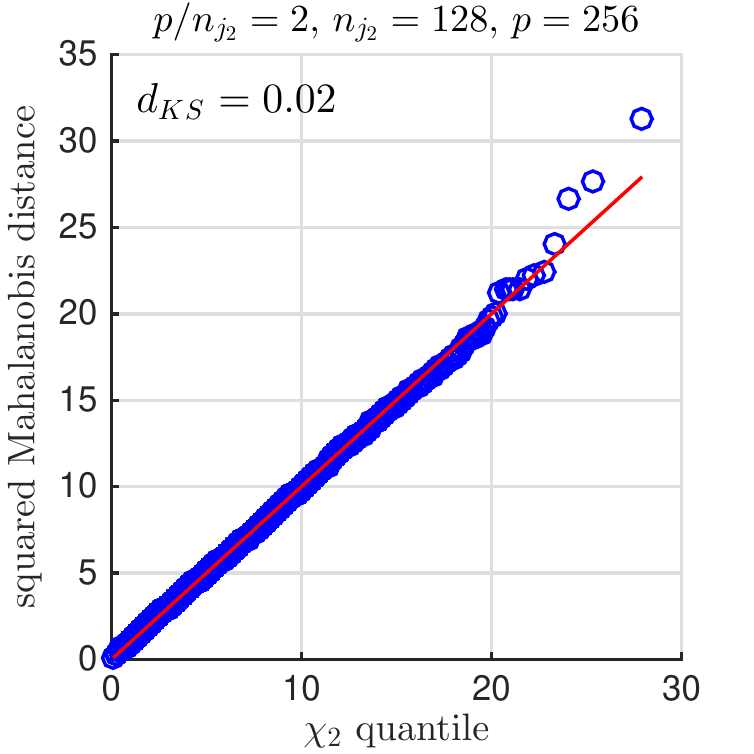}\\
\end{tabular}
\caption[The asymptotic (joint) distribution of $\{\widehat{h}_q\}_{q=1,\hdots,r}$ in high dimensions]{\label{fig:asympt_h-hat_qqplot}{\small {\textbf{The asymptotic (joint) distribution of $\{\widehat{h}_q\}_{q=1,\hdots,r} = \{\widehat{\ell}_{p-r+q}\}_{q=1,\hdots,r}$ in high dimensions.}} Each plot displays Gamma plots (see text in Section \ref{s:MC} for a definition) based on 5000 independent realizations of $\{\widehat{h}_{q}\}_{q=1,\hdots,6}$ for increasing sample size (from top to bottom, respectively) and $p/n_{j_2}=1/2,\,1,\,2$ (left to right column, respectively). The plots further contain Kolmogorov-Smirnov test decisions $d_{KS}$ for the null hypothesis that the Mahalanobis distance follows a $\chi^2_6$ distribution (obtained as averages over $100$ random subsets of size $1250$ of the $5000$ realizations). The plots show that, for any ratio $p/n_{j_2}$ considered, as the sample size increases the distribution of $\{\widehat{h}_{q}\}_{q=1,\hdots,6}$ becomes statistically indistinguishable from a joint Gaussian distribution.  } }
\end{center}
\end{figure}

\clearpage

\begin{figure}
\begin{center}
$p/n_{j_2}=1/2$:\ \\ \ \\

\includegraphics[width=0.85\linewidth]{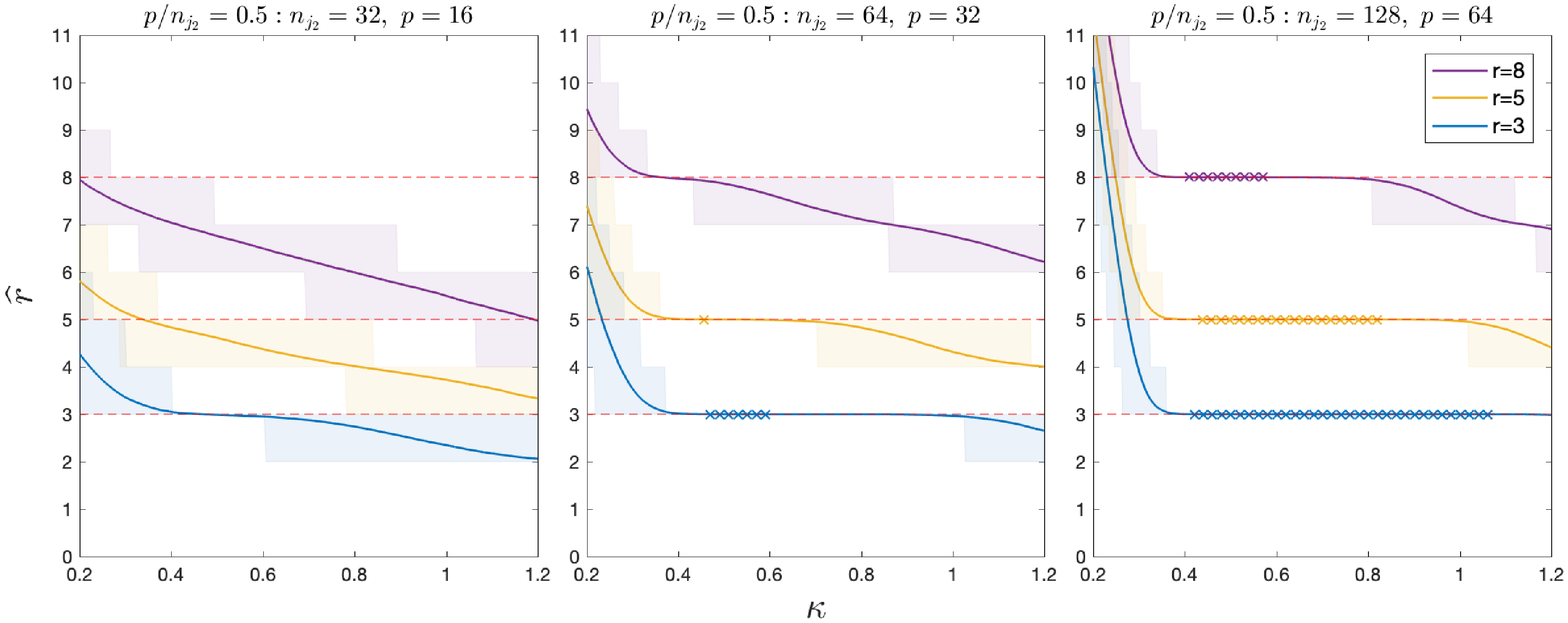}
\end{center}
\begin{center}
$p/n_{j_2}=1$:\ \\ \ \\

\includegraphics[width=0.85\linewidth]{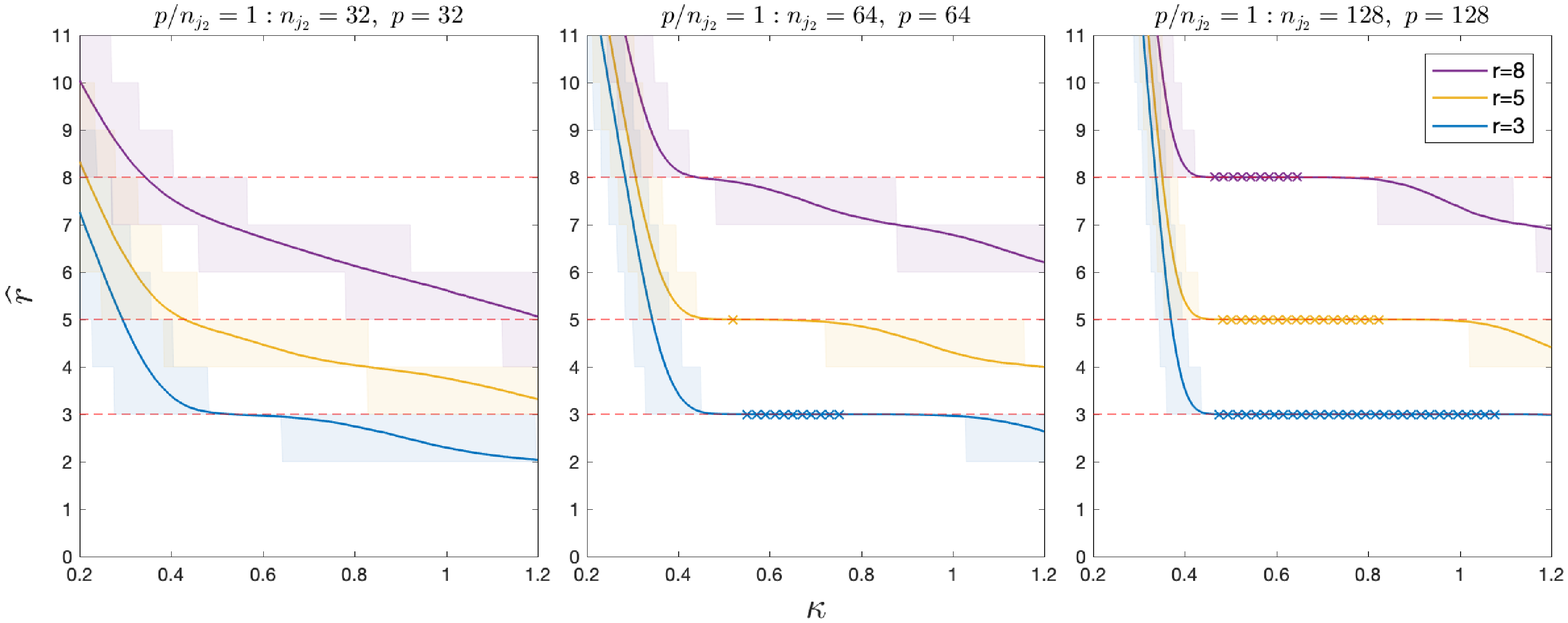}
\end{center}
\begin{center}
$p/n_{j_2}=2$:\ \\ \ \\

\includegraphics[width=0.85\linewidth]{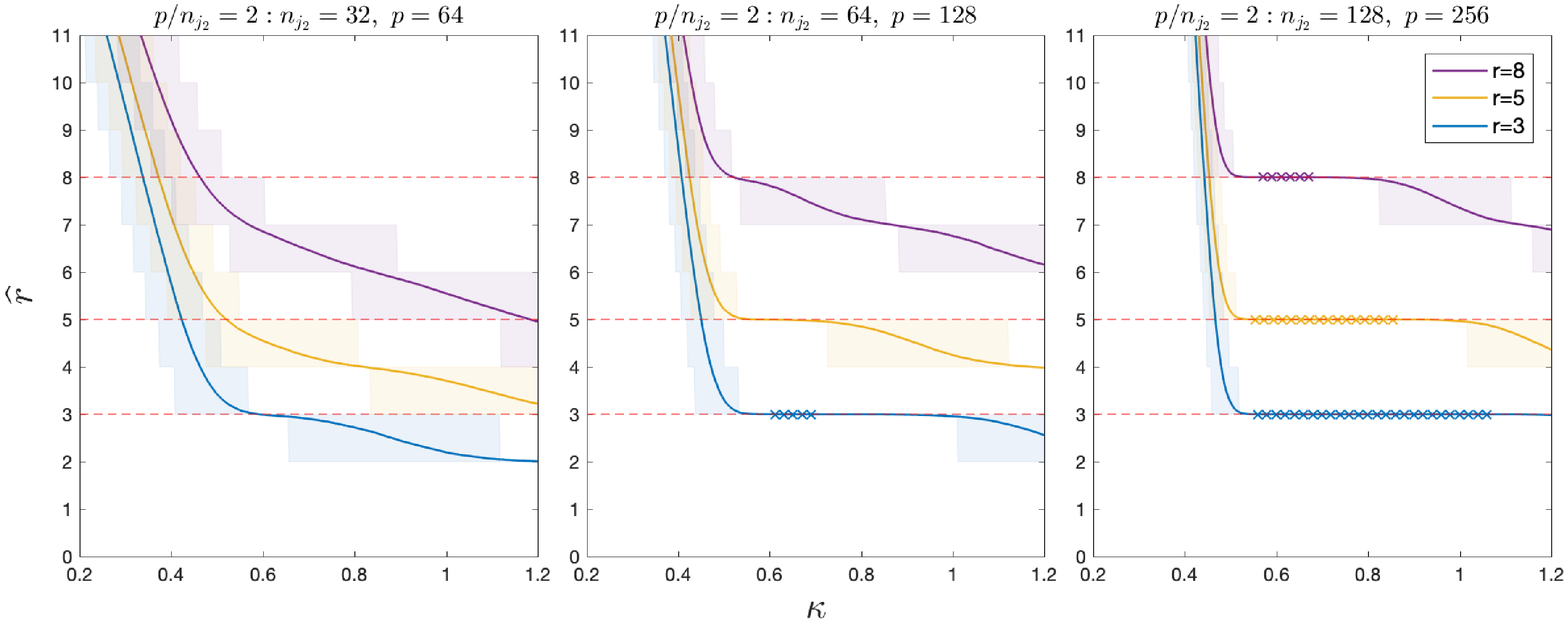}
\end{center}
\caption[The asymptotic performance of $\widehat{r}$]{\label{fig:r-hat}{\small {\textbf{The asymptotic performance of $\widehat{r}(j_1,j_2,\kappa)$.}} Each plot displays $\widehat r$ versus $\kappa$ based on 5000 observations of $\{\Delta_1(j_1,j_2),\ldots, \Delta_{p}(j_1,j_2)\}$ for $p/n_{j_2} \in \{0.5, 1, 2\}$ (top to bottom, respectively), over sample sizes $n \in\{2^{10},2^{12},2^{14}\}$ (left to right column, respectively). The solid lines represent the Monte Carlo average for the particular choice of $r,n$ and $p$. The upper/lower limits of shaded regions correspond the empirical regions between quantiles 0.05 and 0.95. The marked lines (with small ``$x$'' markers) are the points at which the mean was \emph{exactly} equal to the corresponding value of $r$.}}
\end{figure}

\appendix

\section{Assumptions on the wavelet multiresolution analysis}\label{s:MRA_assumptions}

In the main results of the paper, we make use of the following conditions on the underlying wavelet MRA.

\medskip

\noindent {\sc Assumption $(W1)$}: $\psi \in L^2(\bbR)$ is a wavelet function, namely, it satisfies the relations
\begin{equation}\label{e:N_psi}
\int_{\bbR} \psi^2(t)dt = 1 , \quad \int_{\bbR} t^{p}\psi(t)dt = 0, \quad p = 0,1,\hdots, N_{\psi}-1, \quad \int_{\bbR} t^{N_{\psi}}\psi(t)dt \neq 0,
\end{equation}
for some integer (number of vanishing moments) $N_{\psi} \geq 1$.

\medskip

\noindent {\sc Assumption ($W2$)}: the scaling and wavelet functions
\begin{equation}\label{e:supp_psi=compact}
\textnormal{$\phi\in L^1(\bbR)$ and $\psi\in L^1(\bbR)$ are compactly supported }
\end{equation}
and $\widehat{\phi}(0)=1$.

\medskip

\noindent {\sc Assumption $(W3)$}: there is $\alpha > 1$ such that
\begin{equation}\label{e:psihat_is_slower_than_a_power_function}
\sup_{x \in \bbR} |\widehat{\psi}(x)| (1 + |x|)^{\alpha} < \infty.
\end{equation}

\medskip

\noindent {\sc Assumption ($W4$)}: the function
\begin{equation}\label{e:sum_k^m_phi(.-k)}
\sum_{k\in \mathbb{Z}}k^m\phi(\cdot-k)
\end{equation}
is a polynomial of degree $m$ for all $m=0,\hdots,N_{\psi}-1$.

\medskip

Conditions \eqref{e:N_psi} and \eqref{e:supp_psi=compact} imply that $\widehat{\psi}(x)$ exists, is everywhere differentiable and its first $N_{\psi}-1$ derivatives are zero at $x = 0$. Condition \eqref{e:psihat_is_slower_than_a_power_function}, in turn, implies that $\psi$ is continuous (see Mallat \cite{mallat:1999}, Theorem 6.1) and, hence, bounded.

Note that assumptions ($W1-W4$) are closely related to the broad wavelet framework for the analysis of $\kappa$-th order ($\kappa \in \bbN \cup \{0\}$) stationary increment stochastic processes laid out in Moulines et al.\ \cite{moulines:roueff:taqqu:2007:Fractals,moulines:roueff:taqqu:2007:JTSA,moulines:roueff:taqqu:2008} and Roueff and Taqqu~\cite{roueff:taqqu:2009}. The Daubechies scaling and wavelet functions generally satisfy ($W1-W4$) (see Moulines et al.\ \cite{moulines:roueff:taqqu:2008}, p.\ 1927, or Mallat \cite{mallat:1999}, p.\ 253). Usually, the parameter $\alpha$ increases to infinity as $N_{\psi}$ goes to infinity (see Moulines et al.\ \cite{moulines:roueff:taqqu:2008}, p.\ 1927, or Cohen \cite{cohen:2003}, Theorem 2.10.1).

\section{The asymptotic and large scale behavior of the eigenvalues of wavelet random matrices}\label{s:probpaper-results}

In this section, we recap the main results in Abry et al.\ \cite{abry:boniece:didier:wendt:2022:wavelet_eigenanalysis}. It was established that, after proper rescaling, the $r$ largest eigenvalues of a wavelet random matrix $\mathbf{W}(a(n)2^j)$ in high dimensions converge in probability to deterministic functions $\xi_q(2^j)$, $q = 1,\hdots,r$. Thus, such functions can be interpreted as \textit{asymptotic rescaled eigenvalues}. Notably, they display a scaling property. Moreover, the remaining $p(n)-r$ eigenvalues of a wavelet random matrix are bounded in probability.
\begin{theorem}\label{t:lim_n_a*lambda/a^(2h+1)}
(a consequence of Theorem 3.1 in Abry et al.\ \cite{abry:boniece:didier:wendt:2022:wavelet_eigenanalysis}) Assume $(W1-W4)$ and $(A1-A5)$ hold, and fix $j \in \bbN$. %
Then, for $p = p(n)$, the limits
\begin{equation}\label{e:lim_n_a*lambda/a^(2h+1)}
\plim_{n \rightarrow \infty}\frac{\lambda_{p-r+q}\big(\mathbf{W}(a(n)2^j)\big)}{a(n)^{2h_q+ 1}} =: \xi_q(2^j)>0, \quad q=1,\ldots,r,
\end{equation}
exist, and the deterministic functions $\xi_q$ satisfy the scaling relation
\begin{equation}\label{e:xi-i0_scales}
\xi_{q}(2^j)  = 2^{j \hspace{0.5mm}(2 h_{q}+1)} \xi_{q}(1).
\end{equation}
In addition,
\begin{equation}\label{e:lim_lambda_p-r(W)}
\lambda_{1}\big(\mathbf{W}(a(n)2^j)\big)\leq \ldots \leq \lambda_{p-r}\big(\mathbf{W}(a(n)2^j)\big) = O_{\bbP}(1).
\end{equation}
\end{theorem}

Moreover, the asymptotic normality of the $r$ largest wavelet log-eigenvalues in high dimensions was also established. The statement covers the central cases of simple ($h_1 < \hdots < h_r$) and of identical ($h_1 = \hdots = h_r$) scaling eigenvalues. Note that, by comparison to Theorem \ref{t:lim_n_a*lambda/a^(2h+1)}, the asymptotic normality of wavelet log-eigenvalues requires some additional condition so as to ensure the finite-sample differentiability of the eigenvalues with respect to scaling parameters (see also Remark \ref{r:asympt_rescaled_eigenvalues}).
\begin{theorem}\label{t:asympt_normality_lambdap-r+q}
(a consequence of Theorem 3.2 in Abry et al.\ \cite{abry:boniece:didier:wendt:2022:wavelet_eigenanalysis}) Under the assumptions of Theorem \ref{t:h-hat_asympt_normality}, for $p = p(n)$,
\begin{equation}\label{e:asympt_normality_lambda2}
\bbR^{m \cdot r} \ni \Big( \sqrt{n_{a,j}}\Big( \log \lambda_{p-r+q}\big({\mathbf W}(a(n)2^{j})\big)  - \log \lambda_{p-r+q}\big(\bbE {\mathbf W} (a(n)2^{j})\big) \Big)_{q=1,\hdots,r} \Big)_{j=j_1,\hdots,j_2} \stackrel{d}\rightarrow {\mathcal N}(0,\Sigma_{\lambda})
\end{equation}
as $n \rightarrow \infty$  for some $\Sigma_{\lambda} \in {\mathcal S}_{\geq 0}(m \cdot r,\bbR)$. %
\end{theorem}
\begin{remark}
Note that Assumption $(A4)$ (see \eqref{e:p(n),a(n)_conditions} and \eqref{e:varpi_parameter}) implies the analogous Assumption ($A4)$ of Abry et al.\ \cite{abry:boniece:didier:wendt:2022:wavelet_eigenanalysis}, which in turn is assumed in Theorems 3.1 and 3.2 in Abry et al.\ \cite{abry:boniece:didier:wendt:2022:wavelet_eigenanalysis}. The slightly stronger conditions \eqref{e:p(n),a(n)_conditions} and \eqref{e:varpi_parameter}, as well as \eqref{e:<p1,p2>=c12_2}, are needed in Theorem \ref{t:h-hat_asympt_normality} in this paper, i.e., in the context of the wavelet eigenvalue regression.
\end{remark}

\section{Proofs}\label{s:proof_main_results}

In this section, we provide the proofs the main results of the paper. We often use the following notation. For an arbitrary $q\in \{1,\ldots,r\}$, $\mathcal{I}_-$, $\mathcal{I}_0$ and $\mathcal{I}_+$ are index sets given by the relations
\begin{equation} \label{e:def_indexsets}
 \begin{array}{lll}
\mathcal{I}_- := \{\ell: h_\ell < h_q\}, \quad \mathcal{I}_0 := \{\ell: h_\ell =h_q\}, \quad \mathcal{I}_+ := \{\ell: h_\ell > h_q\}.
 \end{array}
 \end{equation}
Note that $\mathcal{I}_-$ and $\mathcal{I}_+$ are possibly empty. Also write
\begin{equation}\label{e:r1,r2,r3}
r_1:= \textnormal{card}(\mathcal{I}_-), \quad r_2:= \textnormal{card}(\mathcal{I}_0) \geq 1, \quad r_3:= \textnormal{card}(\mathcal{I}_+).
\end{equation}
Throughout this section, for notational simplicity, we write
\begin{equation}\label{e:P(n)PH_equiv_P}
P(n)P_H \equiv P(n) \equiv P,
\end{equation}
whose column vectors are denoted by ${\mathbf p}_{\ell}(n) \equiv {\mathbf p}_{\ell}$, $\ell = 1,\hdots,r$. Also for notational convenience, we may write $a\equiv a(n)$.\\

\noindent {\sc Proof of Theorem  \ref{t:h-hat_consistency}}: By Theorem \ref{t:lim_n_a*lambda/a^(2h+1)}, by the properties \eqref{e:sum_wj=0,sum_jwj=1} of the regression weights $w_j$, $j = 1,\hdots,m$, and by relation \eqref{e:xi-i0_scales},
$$
 \widehat h_q = -\frac{1}{2} + \frac{1}{2}\sum_{j=j_1}^{j_2} w_j \log_2 \lambda_{p-r+q}({\mathbf W}(a(n)2^j)) = -\frac{1}{2} + \frac{1}{2}\sum_{j=j_1}^{j_2} w_j  \log_2 \Big(\lambda_{p-r+q}({\mathbf W}(a(n)2^j))/a(n)^{2 h_q+1} \Big)
$$
$$
\stackrel \bbP \to -\frac{1}{2} + \frac{1}{2}\sum_{j=j_1}^{j_2} w_j \log_2 \xi_q(2^j) = -\frac{1}{2} +\frac{1}{2} \sum_{j=j_1}^{j_2} w_j \log_2\Big(2^{j \hspace{0.5mm}(2 h_{q}+1)} \xi_{q}(1)\Big)= h_q. \quad \Box\\
$$

The proof of Theorem \ref{t:h-hat_asympt_normality} mainly relies on Proposition \ref{p:|lambdaq(EW)-xiq(2^j)|_bound_2}, stated and proved next. The proposition establishes not only that the $r$ largest rescaled eigenvalues of the deterministic matrix $\E \W(a2^j)$ converge to their respective asymptotic rescaled eigenvalues $\xi_q(2^j)$ (cf.\ expression \eqref{e:lim_n_a*lambda/a^(2h+1)}), but also it provides an upper bound on the associated convergence rate.

\begin{proposition}\label{p:|lambdaq(EW)-xiq(2^j)|_bound_2}
Fix $j \in \bbN$ and suppose conditions $(W1-W4)$ and $(A1-A5)$ hold. Suppose that either
\begin{itemize}
\item[(i)] $0<h_1<\ldots<h_r<1$; or
\item[(ii)] $h_1=\ldots=h_r,$ and whenever $q_1\neq q_2$, the functions $\xi_q(2^j)$ in \eqref{e:lim_n_a*lambda/a^(2h+1)} satisfy
\begin{equation}\label{e:xiq_distinct_2}
\xi_{q_1}(1)\neq \xi_{q_2}(1).
\end{equation}
\end{itemize}
Then, for some $C > 0$ that does not depend on $j$,
\begin{equation}\label{e:|lambdaq(EW)-xiq(2^j)|_bound}
\Big|\frac{\lambda_{p-r+q}(\bbE {\mathbf W}(a(n)2^j))}{a(n)^{2h_q+1}} - \xi_q(2^j)\Big| \leq \frac{C}{a(n)^{ \varpi}}, \quad q = 1,\hdots,r.
\end{equation}
\end{proposition}

\begin{proof}
Assume condition $(i)$ holds. We prove only the statement for $q\in\{2,\ldots,r-1\}$ since the cases $q=1,r$ can be handled by a simplified version of the same argument.
For
$$
{\boldsymbol \gamma} = \big(\gamma_1,\hdots,\gamma_r \big)^* \in \bbR^r, \quad \mathbf x \in \bbR^{r_3},
$$
consider the deterministic function
\begin{equation}\label{e:def_gn_deterministic}
g_{n}(\boldsymbol \gamma ;\mathbf{x}):=\mathbf y^*_n(\boldsymbol \gamma;\mathbf{x}) \B_a(2^j) \mathbf y_n(\boldsymbol \gamma;\mathbf{x}) \in \bbR,
\end{equation}
where
$$
{\mathbf y}_n(\boldsymbol \gamma ;\mathbf{x})=\Big(\underbrace{\gamma_1 a(n)^{h_1-h_q},\ldots,\gamma_{r_1} a(n)^{h_{r_1}-h_q}}_{r_1},\gamma_{q},\underbrace{x_{q+1},\ldots,x_{r}}_{r_3}\Big)^*.
$$
Now let
\begin{equation}\label{e:g(gamma-q;x-vec)}
g(\boldsymbol \gamma;\mathbf{x}) =  {\mathbf y}^*(\boldsymbol \gamma ;\mathbf{x}) \B(2^j) {\mathbf y}(\boldsymbol \gamma ;\mathbf{x}) \in \bbR,
\end{equation}
where
$$
\bbR^{r} \ni {\mathbf y}( \boldsymbol \gamma ;\mathbf{x}) = \Big(\underbrace{0,\ldots,0}_{r_1}, \gamma_{q},\underbrace{x_{q+1},\ldots,x_{r}}_{r_3}\Big)^* = \lim_{n \rightarrow \infty}{\mathbf y}_n(\boldsymbol \gamma ;\mathbf{x}).
$$
(observe that $\mathbf y(\boldsymbol \gamma; \mathbf x)$ depends on $\boldsymbol \gamma$ only through its $q$--th entry $\gamma_{q}$). For $\ell=1,\ldots,p$, let $\mathbf u_{\ell}(n)$ denote a unit  eigenvector of $\E \W(a2^j)$ associated with its $\ell$--th eigenvalue in nondecreasing order. Define
\begin{equation}\label{e:gamma-q(n)}
\boldsymbol \gamma_q(n):= P^*(n)\mathbf u_{p-r+q}(n).
\end{equation}
By Proposition \ref{p:|lambdaq(EW)-xiq(2^j)|_bound},  $\mathbf u_{p-r+q}(n)$ can be chosen so that the limit
\begin{equation}\label{e:lim_gamma-q(n)}
\lim_{n\to \infty}\boldsymbol \gamma_q(n)=:\boldsymbol \gamma_q=(\gamma_{1,q},\ldots,\gamma_{r,q})^*, \quad n\to\infty,
\end{equation}
exists. So, let
\begin{equation}\label{e:x_*(n),x_*}
\mathbf x_*(n)\in \bbR^{r_3} \quad \textnormal{and} \quad \mathbf x_*=(x_{*,q+1},\ldots,x_{*,r})\in \bbR^{r_3}
\end{equation}
be the (unique) minimizers of the functions $g_n(\boldsymbol \gamma_q(n);\cdot)$ and $g( \boldsymbol \gamma_q ;\cdot)$, respectively, where such functions are given in \eqref{e:def_gn_deterministic} and \eqref{e:g(gamma-q;x-vec)}. Observe that, as $n\to\infty$, $\mathbf x_*(n)\to \mathbf x_*$, implying $\mathbf y_n(\boldsymbol \gamma_q(n),\mathbf x_*(n)) \to \mathbf y(\boldsymbol \gamma_q,\mathbf x_*)$. Hence,
\begin{equation}\label{e:gn_to_g}
g_n(\boldsymbol \gamma_q(n),\mathbf x_*(n))\to g(\boldsymbol \gamma_q,\mathbf x_*),\quad n\to\infty.
\end{equation}
By Lemma \ref{l:<p,w>=infinitesimal},  for large enough $n$ we may take a sequence of unit vectors $\mathbf v(n) \in \text{span}\{{\mathbf u}_{p-r+q}(n),\ldots,{\mathbf u}_{p}(n)\}$ such that
\begin{equation}\label{e:<p-ell(n),v(n)>}
\langle \p_\ell(n),\mathbf v(n)\rangle = \frac{x_{*,\ell}}{a^{h_\ell-h_q}},\quad \ell= q+1,\ldots,r, \quad P^*(n)\mathbf v(n)\to \boldsymbol\gamma_q.
\end{equation}
Therefore, as $n \rightarrow \infty$,
$$
a^{\mathbf h-h_q I}P^*(n)\mathbf v(n)=\mathbf y_n\big(P^*(n)\mathbf v(n),\mathbf x_*(n)\big)  \to {\mathbf y}( \boldsymbol \gamma_q ;\mathbf{x}_*),
$$
which implies
\begin{equation}\label{e:gn_to_g_2}
g_n(P^*(n)\mathbf v(n),\mathbf x_*(n))\to g(\boldsymbol \gamma_q,\mathbf x_*),\quad n\to\infty.
\end{equation}
Moreover, let
$$
\bbR^{r_3}\ni \hspace{0.5mm} \mathbf x(n): = \big(\langle \p_{q+1}(n),{\mathbf u}_{p-r+q}(n)\rangle a^{h_{q+1}-h_q},\ldots,\langle \p_{r}(n),{\mathbf u}_{p-r+q}(n)\rangle a^{h_{r}-h_q}\big), \quad n \in \bbN
$$
(not to be confused with the minimizer $\mathbf x_*(n)$ of $g_n(\boldsymbol \gamma_q(n),\cdot)$ as in \eqref{e:x_*(n),x_*}). Then, we can express
$$
\lambda_{p-r+q}\Big( \frac{\E\W(a2^j)}{a^{2h_q+1}} \Big)=g_{n}(\boldsymbol \gamma_q(n),\mathbf x(n)) + {\mathbf u}^*_{p-r+q}(n) \frac{\E \W_Z(a2^j)}{a^{2h_q+1}} {\mathbf u}_{p-r+q}(n).
$$
Thus, for all large $n$ and for ${\boldsymbol \gamma}_q(n)$ and $\mathbf v(n)$ as in \eqref{e:gamma-q(n)} and \eqref{e:<p-ell(n),v(n)>}, respectively,
\begin{equation}\label{e:g_n<lambda<g_n}
g_{n}(\boldsymbol \gamma_q(n),{\mathbf{x}}_*(n)) + {\mathbf u}^*_{p-r+q}(n) \frac{\E \W_Z(a2^j)}{a^{2h_q+1}} {\mathbf u}_{p-r+q}(n)
$$
$$
\leq g_{n}(\boldsymbol \gamma_q(n),\mathbf x(n)) +{\mathbf u}^*_{p-r+q}(n) \frac{\E \W_Z(a2^j)}{a^{2h_q+1}} {\mathbf u}_{p-r+q}(n)
$$
$$
={\mathbf u}^*_{p-r+q}(n) \frac{\bbE {\mathbf W}(a2^j)}{a^{2h_q+1}} {\mathbf u}_{p-r+q}(n) = \frac{\lambda_{p-r+q}(\bbE {\mathbf W}(a2^j))}{a^{2h_q+1}}
$$
$$
\leq \mathbf v^*(n) \frac{ \bbE \mathbf{W}(a2^j) }{a^{2h_q + 1}}\mathbf v(n)
= g_{n}(\mathbf v(n);\mathbf x_*)+  \mathbf v^*(n) \frac{\E \W_Z(a2^j)}{a^{2h_q+1}} \mathbf v(n),
\end{equation}
where in the second inequality we used the fact that $\mathbf v(n) \in \text{span}\{{\mathbf u}_{p-r+q}(n),\ldots,{\mathbf u}_{p}(n)\}$. However, note that $\|\E \W_Z(a2^j)/a^{2h_q+1}\| =o(a^{-\varpi})$ under assumption \eqref{e:assumptions_WZ=OP(1)}. Then, by Corollary \ref{c:PWXP^*+R_asymptotics}  with $M_n = \bbE\W_Z(a2^j)/a^{2h_q+1}$,
\begin{equation}\label{e:lambda_p-r+q_rescaled->xi_q}
 \frac{\lambda_{p-r+q}(\bbE {\mathbf W}(a(n)2^j))}{a(n)^{2h_q+1}} \to \xi_q(2^j), \quad n \rightarrow \infty.
\end{equation}
Therefore, in view of \eqref{e:gn_to_g} and \eqref{e:gn_to_g_2}, by taking limits in \eqref{e:g_n<lambda<g_n}, we see that
\begin{equation}\label{e:xi_q(2^j)=g(gamma_q,x*)}
\xi_q(2^j) = g(\boldsymbol \gamma_q,\mathbf x_*).
\end{equation}
Moreover, again since $\|\E \W_Z(a2^j)/a^{2h_q+1}\| =o(a^{-\varpi})$ under assumption \eqref{e:assumptions_WZ=OP(1)}, expressions \eqref{e:g_n<lambda<g_n}, \eqref{e:lambda_p-r+q_rescaled->xi_q} and \eqref{e:xi_q(2^j)=g(gamma_q,x*)} imply that
$$
\Big|\frac{\lambda_{p-r+q}(\bbE {\mathbf W}(a(n)2^j))}{a(n)^{2h_q+1}} - \xi_q(2^j)\Big|
$$
\begin{equation}\label{e:xiq_lambdaq_maxbound}
\leq  \max \Big\{\Big|g_{n}\big(\boldsymbol \gamma_q(n),{\mathbf{x}}_*(n)\big) -g\big(\boldsymbol \gamma_q,\mathbf x_* \big) \Big| ,\Big| g_{n}\big(\mathbf v(n);\mathbf x_* \big)- g\big(\boldsymbol \gamma_q,\mathbf x_* \big)\Big|\Big\} + o(a^{-\varpi}).
\end{equation}
Consider the first and the second terms inside the $\max\{\cdot,\cdot\}$ on the right-hand side of expression \eqref{e:xiq_lambdaq_maxbound}. If
\begin{equation}\label{e:|g_{n}gamma_q,x_*(n))-xi_q(2^j)| = O(a^{-varpi})}
\textnormal{(a)} \quad \Big|g_{n}(\boldsymbol \gamma_q(n),\mathbf{x}_*(n)) -\xi_q(2^j) \Big| = O(a^{-\varpi})
\end{equation}
and
\begin{equation}\label{e:|g_{n}(v(n);x_{*})-xi_q(2^j)|=O(a^(-varpi))}
\textnormal{(b)} \quad \Big| g_{n}(\mathbf v(n);\mathbf{x}_{*}) -\xi_q(2^j)\Big|
 = O(a^{-\varpi}),
\end{equation}
then \eqref{e:|lambdaq(EW)-xiq(2^j)|_bound} holds under condition $(i)$. So, we now establish (a) and (b).

First, we show (a). Let ${\boldsymbol \gamma}_q(n)$ and ${\mathbf x}_*(n)$ be as in \eqref{e:gamma-q(n)} and \eqref{e:x_*(n),x_*}, respectively. For notational simplicity, write
$$
{\mathbf y}(n) = \big(y_1(n),\hdots,y_r(n)\big)^* := \mathbf y_{n}\big(\boldsymbol \gamma_q(n),{\mathbf{x}}_*(n)\big)
$$
\begin{equation}\label{e:y_vec(n)}
= \Big( \underbrace{\langle {\mathbf p}_{1}(n),{\mathbf u}_{p-r+q}(n) \rangle a^{h_1-h_q},\hdots,\langle {\mathbf p}_{r_{1}}(n),{\mathbf u}_{p-r+q}(n) \rangle a^{h_{r_1}-h_q}}_{r_1},\langle {\mathbf p}_{q}(n),{\mathbf u}_{p-r+q}(n) \rangle , \underbrace{{\mathbf x}_{*}(n)}_{r_3}\Big)
\end{equation}
and
\begin{equation}\label{e:y-vec}
{\mathbf y}=\big(\underbrace{0,\ldots,0}_{r_1},\gamma_{qq},\underbrace{{\mathbf x}_{*}}_{r_3}\big)^* =(y_1,\ldots,y_r)^* := \lim_{n\to\infty}\mathbf y(n),
\end{equation}
where $\lim_{n \rightarrow \infty}\langle {\mathbf p}_{q}(n),{\mathbf u}_{p-r+q}(n) \rangle = \gamma_{qq}$ as a consequence of \eqref{e:lim_gamma-q(n)}. Recast
 $$
 g_{n}(\boldsymbol \gamma_q(n),\mathbf{x}_*(n)) -\xi_q(2^j)  = g_{n}(\boldsymbol \gamma_q(n),{\mathbf{x}}_*(n)) -g(\boldsymbol \gamma_q,\mathbf x_*)
 $$
 $$
 =  \mathbf y^*(n)\mathbf{B}_a(2^j) \mathbf y (n) -  \mathbf y^* \B(2^j)  \mathbf y
 $$
\begin{equation}\label{e:xiq-gn_lower}
= (\mathbf y(n)-\mathbf y)^*\mathbf{B}_a(2^j) (\mathbf y(n)-\mathbf y) + 2 \mathbf y^* \B(2^j)  (\mathbf y(n)-\mathbf y) + {\mathbf y}^*\big({\mathbf B}_a(2^j)-{\mathbf B}(2^j)\big){\mathbf y}.
\end{equation}
Therefore, in view of condition \eqref{e:|bfB_a(2^j)-B(2^j)|=O(shrinking)}, if
\begin{equation}\label{e:|y(n)-y|=O(a^{-varpi})}
\| \mathbf y(n)-\mathbf y \|=O(a^{-\varpi}),
\end{equation}
then
$$
\big| g_{n}(\boldsymbol \gamma_q(n),{\mathbf{x}}_*(n)) -g(\boldsymbol \gamma_q,\mathbf x_*)\big| = O(a^{-\varpi}).
$$
Thus, \eqref{e:|g_{n}gamma_q,x_*(n))-xi_q(2^j)| = O(a^{-varpi})} holds, which establishes (a). So, we now show \eqref{e:|y(n)-y|=O(a^{-varpi})}. We establish relation \eqref{e:|y(n)-y|=O(a^{-varpi})} entry-wise for each of the ranges $\ell\in\mathcal{I}_-=\{1,\ldots,q-1\}$, $\ell = q$ and $\ell\in\mathcal{I}_+=\{q+1,\ldots,r\}$.

In fact, for $\ell\in\mathcal{I}_- = \{1,\ldots,q-1\}$, expression \eqref{e:y_vec(n)} shows that
 \begin{equation}\label{e:y(n)-z_-}
 | y_\ell(n)- y_\ell|=| y_\ell(n)|= O\Big(a^{h_\ell-h_q}\Big)= O(a^{-\varpi}).
 \end{equation}

On the other hand, for $\ell=q$, let $P(n)=Q(n)R(n)$ be the $QR$ decomposition of $P(n)$, where $R(n)\in GL(r,\bbR)$, $Q(n)=(\widetilde \p_1(n),\ldots,\widetilde \p_r(n))\in M(p,r,\bbR)$ with orthonormal columns. As in \eqref{e:sum<pi(n),up-r+q(n)>2-o(a-varpi)}, let $\p_{r+1}(n),\ldots,\p_p(n)$ be an orthonormal basis for the nullspace of $P^*(n)$. Observe that, by Lemma \ref{l:p-tilde_upk=o(a-w)},
  \begin{equation}\label{e:1-<p-tilde_k(n),u_p-r+k(n)>^2=O(small)}
  1- \langle \widetilde \p_k(n), {\mathbf u}_{p-r+k}(n)\rangle^2=O(a^{-2\varpi}), \quad k=1,\ldots, r.
  \end{equation}
Consequently, after flipping the sign of ${\mathbf u}_{p-r+q}(n)$ if necessary, $\|Q^*(n){\mathbf u}_{p-r+q}(n)- \mathbf e_q\| = O(a^{-\varpi})$. Recalling \eqref{e:gamma-q(n)} and the condition that $R(n)\to R$ as $n \rightarrow \infty$ {(see \eqref{e:<p1,p2>=c12_2})},
$$
\|{\boldsymbol \gamma}_q(n) - R \mathbf e_q\| = \|R^*(n)Q^*(n){\mathbf u}_{p-r+q}(n) - R \mathbf e_q\|$$
$$\leq \|R(n)- R\| + \|R\| \| Q^*(n){\mathbf u}_{p-r+q}(n)-\mathbf e_q\| = O(a^{-\varpi}),
$$
where in the last equality we make use of condition \eqref{e:<p1,p2>=c12_2}. Hence, ${\boldsymbol \gamma}_q=\lim_{n\to\infty} {\boldsymbol \gamma}_q(n) = R \mathbf e_q$. Keeping in mind expressions \eqref{e:y_vec(n)} and \eqref{e:y-vec}, this implies that
 \begin{equation}\label{e:y(n)-z_0}
|y_q(n)-y_q| = |\langle \p_q(n), {\mathbf u}_{p-r+q}(n)\rangle - \gamma_{qq}| = O(a^{-\varpi}).
\end{equation}
In other words, \eqref{e:|y(n)-y|=O(a^{-varpi})} also holds for $\ell = q$.

Turning to $y_\ell(n)-y_\ell$ for $\ell\in\mathcal{I}_+=\{q+1,\ldots,r\}$, recall that
$$
\mathbf x_*(n) := (x_{*,q+1}(n),\ldots,x_{*,r}(n))^*
$$
is the unique minimizer of $g_n(\boldsymbol \gamma_q(n);\cdot)$ in \eqref{e:def_gn_deterministic}, and that $\mathbf x_*=\lim_{n\to\infty}\mathbf x_*(n)$.  Now, reexpress
\begin{equation}\label{e:B(2^j)=(B_i,ell)}
 {\mathbf B}_a(2^j) =( {\mathbf B}_{i\ell}(n))_{i,\ell=1,2,3}, \quad {\mathbf B}(2^j) =(\mathbf B_{i\ell})_{i,\ell=1,2,3},
\end{equation}
where ${\mathbf B}_{i\ell}(n)$ and $ {\mathbf B}_{i\ell}$ denote blocks of size $r_i\times r_\ell$.
The first order conditions for the minimization of $g_n(\boldsymbol \gamma_q(n);\cdot)$  imply that
\begin{equation}\label{e:B0,B+,firstorder_cond}
\begin{pmatrix}
\mathbf B _{31}(n) & \mathbf B_{32}(n) & \mathbf B_{33}(n)
\end{pmatrix} \begin{pmatrix}
O(a^{-\varpi}) \\
\langle {\mathbf p}_{q}(n),{\mathbf u}_{p-r+q}(n) \rangle \\
\mathbf x_*(n)
\end{pmatrix}= \mathbf 0
\end{equation}
where $O(a^{-\varpi}) \in \bbR^{r_1}$ contains the first $r_1$ terms in $\mathbf y(n)$ (see \eqref{e:y_vec(n)}).  Since $\B_{33}(n)\in\mathcal{S}_{>0}(r_3,\bbR)$ for all large $n$, %
by rearranging \eqref{e:B0,B+,firstorder_cond}  and multiplying by $\B_{33}(n)^{-1}$ on the left we have, for all large $n$,
$$
\mathbf{x}_*(n) = -\B_{33}(n)^{-1}\Big( O(a^{-\varpi}) + \B_{32}(n)\langle {\mathbf p}_{q}(n),{\mathbf u}_{p-r+q}(n) \rangle\Big),  \quad \mathbf{x}_* = -\B_{33}^{-1}\B_{32} \gamma_{qq}.
$$
Together with \eqref{e:y(n)-z_-} and \eqref{e:y(n)-z_0}, as well as assumption \eqref{e:|bfB_a(2^j)-B(2^j)|=O(shrinking)}, this implies that
$$
\|(y_\ell(n)-y_\ell)_{\ell\in\mathcal I_+}\|=\|\mathbf{x}_*(n)-\mathbf x_*\|  \leq  O(a^{-\varpi}) + \| \B_{33}(n)^{-1}\B_{32}(n)\langle {\mathbf p}_{q}(n),{\mathbf u}_{p-r+q}(n)\rangle - \B_{33}^{-1}\B_{32} \gamma_{qq}\| %
$$
$$
\leq O(a^{-\varpi})+ \big\| (\B_{33}(n)^{-1}\B_{32}(n)-\B_{33}^{-1}\B_{32}\big)\langle {\mathbf p}_{q}(n),{\mathbf u}_{p-r+q}(n)\rangle\big\| +  \big\|\B_{33}^{-1}\B_{32}\big(\langle {\mathbf p}_{q}(n),{\mathbf u}_{p-r+q}(n)\rangle-\gamma_{qq}\big)\big\|
$$
$$
\leq O(a^{-\varpi})+ C \Big(\big\| (\B_{33}(n)^{-1}\B_{32}(n)-\B_{33}^{-1}\B_{32}\big)\big\| +  \big|\langle {\mathbf p}_{q}(n),{\mathbf u}_{p-r+q}(n)\rangle-\gamma_{qq}\big|\Big) = O(a^{-\varpi}).
$$
Hence, \eqref{e:|y(n)-y|=O(a^{-varpi})} also holds for $\ell \in {\mathcal I}_+$. Thus, \eqref{e:|y(n)-y|=O(a^{-varpi})} holds for $\ell = 1,\hdots,r$. This establishes \eqref{e:|g_{n}gamma_q,x_*(n))-xi_q(2^j)| = O(a^{-varpi})} and, hence, (a).

Now, we turn to (b). Consider the sequence of unit vectors $\mathbf v(n)$ given by \eqref{e:<p,w>=infinitesimal_2} in Lemma \ref{l:<p,w>=infinitesimal} for the deterministic matrix $\E \W(a2^j)$.  For a sequence of scalars $\{c_{q}(n), c_{q+1}(n),\ldots,c_r(n)\}$ satisfying $\sum_{i=q}^rc_i^2(n) = 1$, we may write $\mathbf v(n) =  \sum_{i=q}^rc_{i}(n){\mathbf u}_{p-r+i}(n)$. By expression \eqref{e:<p,w>=infinitesimal_2} of Lemma \ref{l:<p,w>=infinitesimal},
$$
1- c_q(n)^2 = \sum_{i=q+1}^r c_i(n)^2 = O(a^{-2\varpi}).
$$
Moreover, by flipping the sign of $\mathbf u_{p-r+q}(n)$ if necessary, we may suppose $c_q(n)\to 1$ as $n \rightarrow \infty$. Hence, $|1-c_q(n)| = O(a^{-2\varpi})$. Therefore,
\begin{equation}\label{e:(pl,v)-(pl,u)}
|\langle \p_q(n),\mathbf v(n)\rangle -\langle \p_q(n),{\mathbf u}_{p-r+q}(n)\rangle|
$$
$$
= \Big| (c_q(n)-1)\langle \p_q(n),{\mathbf u}_{p-r+q}(n)\rangle +  \sum_{i\in\mathcal{I}_+}c_{i}(n)\langle \p_q(n),{\mathbf u}_{p-r+i}(n)\rangle \Big| $$
$$
\leq |1-c_q(n)|\cdot |\langle \p_q(n),{\mathbf u}_{p-r+q}(n)\rangle| +  O(a^{-\varpi}) =  O(a^{-\varpi}).
\end{equation}
Thus, \eqref{e:y(n)-z_0} and \eqref{e:(pl,v)-(pl,u)} show that
\begin{equation}\label{e:<pl,v>-gamma_lq}
\big|\langle \p_q(n),\mathbf v (n)\rangle - \gamma_{qq}\big| \leq \big|\langle \p_q(n),\mathbf v(n)\rangle -\langle \p_q(n),{\mathbf u}_{p-r+q}(n)\rangle\big| + \big|\langle \p_q(n),{\mathbf u}_{p-r+q}(n)\rangle - \gamma_{qq}\big|
$$
$$
= O(a^{-\varpi}) + O(a^{-\varpi})  = O(a^{-\varpi}).
\end{equation}
Now, define
$$
\widetilde{\mathbf y}(n)=a^{\mathbf{h}-h_q I}P^*(n) \mathbf v(n)=\mathbf y_n\big(P^*(n)\mathbf v(n),\mathbf x_*\big).
$$
Thus, entry-wise, we can express
$$
\widetilde{\mathbf y}(n)_{\ell} =
\begin{cases}
\langle \p_\ell(n),{\mathbf v}(n)\rangle a^{h_\ell-h_q}=O(a^{-\varpi}), & \ell \in \mathcal{I}_-;\\
\langle \p_\ell(n),{\mathbf v}(n)\rangle, & \ell=q;\\
x_{*,\ell}, & \ell \in \mathcal{I}_+.\\
\end{cases}
$$
From \eqref{e:<pl,v>-gamma_lq} and for ${\mathbf y}$ as in \eqref{e:y-vec}, we obtain
\begin{equation}\label{e:|y-tilde(n)-y|^2}
\|\widetilde{\mathbf y}(n)-\mathbf y\|^2= \big( \langle \p_q(n),\mathbf v(n)\rangle - \gamma_{qq} \big)^2 + O(a^{-2\varpi})=O(a^{-2\varpi}).
\end{equation}
Hence,
$$
\Big| g_{n}(\mathbf v(n);\mathbf{x}_{*}) -\xi_q(2^j)\Big| = \Big | \widetilde {\mathbf y}^*(n) \mathbf{B}_a(2^j)\widetilde{ \mathbf y}(n) -  {\mathbf y}^* \mathbf{B}(2^j){ \mathbf y}\Big|
$$
$$
= \Big|(\widetilde{\mathbf y}(n)-\mathbf y)^*\mathbf{B}_a(2^j) (\widetilde{\mathbf y}(n)-\mathbf y) + 2 \mathbf y^* \B(2^j)  (\widetilde{\mathbf y}(n)-\mathbf y) + {\mathbf y}^*\big({\mathbf B}_a(2^j)-{\mathbf B}(2^j)\big){\mathbf y}\Big|
$$
$$
\leq C\|{\mathbf B}_a(2^j)-{\mathbf B}(2^j)\| + C' \|\widetilde {\mathbf y}(n)-\mathbf y\|  = O(a^{-\varpi}),
$$
where the last equality is a consequence of condition \eqref{e:|bfB_a(2^j)-B(2^j)|=O(shrinking)} and of relation \eqref{e:|y-tilde(n)-y|^2}.
This establishes \eqref{e:|g_{n}(v(n);x_{*})-xi_q(2^j)|=O(a^(-varpi))} and, hence, (b). Thus, as anticipated, \eqref{e:|lambdaq(EW)-xiq(2^j)|_bound} holds under condition $(i)$.

We now turn to \eqref{e:|lambdaq(EW)-xiq(2^j)|_bound} for the case $(ii)$. Let
$$
\W_{P X}(a2^j) = P \W_{X}(a2^j) P^*.
$$
Fix any $q\in\{1,\ldots,r\}$. By Weyl's inequality (e.g., Vershynin \cite{vershynin:2018}, Theorem 4.5.3),
$$
\Big|\frac{\lambda_{p-r+q}\big(\bbE\W(a2^j)\big)}{a^{2h_q+1}} - \frac{\lambda_{p-r+q}\big(\E\W_{P X}(a2^j)\big)}{a^{2h_q+1}}\Big|
$$
$$
\leq \|\bbE\W(a2^j)/a^{2h_q+1}-\bbE\W_{P X}(a2^j)/a^{2h_q+1}\| = O(a^{-\varpi}).
$$
So, it suffices to show \eqref{e:|lambdaq(EW)-xiq(2^j)|_bound} for the matrix $\bbE \W_{PX}(a2^j)/a^{2h_q+1}=P \B_a(2^j)P^*$ in place of $\bbE \W(a2^j)$.

Let $\widetilde {\mathbf u}_i(n)$, $i=1,\dots,p$, be an orthonormal basis of eigenvectors of the (deterministic) matrix $P\B_a(2^j)P^*$, where the associated ordering of eigenvalues is arbitrary in the case of ties. Observe that $P^*P\B_a(2^j)P^*\widetilde {\mathbf u}_i(n) = \lambda_i(P\B_a(2^j)P^*) \cdot P^*\widetilde {\mathbf u}_i(n)$ for $i=1,\ldots,p$. In particular,
\begin{equation}\label{e:lambda(P*PB)=lambda(P*BP)}
\lambda_{p-r+q}(P\B_a(2^j)P^*) = \lambda_{q}(P^*P\B_a(2^j)), \quad q = 1,\hdots,r
\end{equation}
(where $P^*\widetilde {\mathbf u}_{p-r+q}(n)$ is the associated eigenvector of the non-symmetric matrix $P^*P\B_a(2^j) \in M(r,\bbR)$). By a similar argument, and considering again the $QR$ decomposition $P \equiv P(n)=Q(n)R(n)$,
\begin{equation}\label{e:lambda(R*(n)R(n)B)=lambda(R*(n)BR(n))}
\lambda_{q}\big(R(n)\B_a(2^j)R^*(n)\big) = \lambda_{q}\big(R^*(n)R(n)\B_a(2^j)\big), \quad q = 1,\hdots,r.
\end{equation}
As a consequence of \eqref{e:lambda(P*PB)=lambda(P*BP)}, \eqref{e:lambda(R*(n)R(n)B)=lambda(R*(n)BR(n))} and of the fact that $Q^*(n)Q(n) = I_r$,
$$
\lambda_{q}\big(P^*(n)P(n)\B_a(2^j)\big) = \lambda_{q}\big(R(n)\B_a(2^j)R^*(n)\big), \quad q = 1,\hdots,r.
$$
So, by taking limits in \eqref{e:lambda(P*PB)=lambda(P*BP)}, Lemma \ref{c:PWXP^*+R_asymptotics} with $M_n\equiv \mathbf 0$ implies that, for $q = 1,\hdots,r$, $\lambda_{q}(P^*(n)P(n)\B_a(2^j))\to \xi_q(2^j)$ as $n \rightarrow \infty$.  Moreover, by a similar argument to the one leading to \eqref{e:lambda(R*(n)R(n)B)=lambda(R*(n)BR(n))},
\begin{equation}\label{e:xi_q(2^j)=lambda-q(AB(2^j))}
\xi_q(2^j) = \lambda_{q}(A\B(2^j))=\lambda_{q}(R\B(2^j)R^*)
\end{equation}
due to assumptions \eqref{e:|bfB_a(2^j)-B(2^j)|=O(shrinking)} and \eqref{e:<p1,p2>=c12_2}. Therefore, again by Weyl's inequality,
$$
\big |\lambda_q\big(P^*(n)P(n) \B_a(2^j)\big) - \xi_q(2^j)\big|= \big|\lambda_q\big(R^*(n)R(n) \B_a(2^j)\big) - \lambda_q(R^*R \B(2^j))\big|
$$
$$
= \big|\lambda_q\big(R(n)\B_a(2^j)R^*(n)\big) - \lambda_q\big(R \B(2^j)R^*\big)\big| \leq \big\| R(n)\B_a(2^j)R^*(n) - R \B(2^j)R^*\big\|
$$
$$
\leq \| (R(n)-R)\B_a(2^j)(R(n)-R)^*\| + \|(R(n)-R)\B_a(2^j)R^*\|
$$
\begin{equation}\label{e:|P^*PB_a(2^j)-AB(2^j)|=O(a^(-varpi))}
+ \|R\B_a(2^j)(R(n)-R)^*\| + \|R(\B_a(2^j)-\B(2^j))R^* \| = O(a^{-\varpi}),
\end{equation}
where the last equality is a consequence of conditions \eqref{e:|bfB_a(2^j)-B(2^j)|=O(shrinking)} and \eqref{e:<p1,p2>=c12_2}.
This establishes \eqref{e:|lambdaq(EW)-xiq(2^j)|_bound} in the case $(ii)$. $\Box$
\end{proof}\\

We are now in a position to prove Theorem \ref{t:h-hat_asympt_normality} (cf.\ Corollary 2, $(ii)$, in Abry and Didier \cite{abry:didier:2018:n-variate}).\\

\noindent {\sc Proof of Theorem \ref{t:h-hat_asympt_normality}}: For a fixed $q \in\{1,\hdots,n\}$, the left-hand side of \eqref{e:h^q_asymptotically_normal} can be recast in the form
$$
\sqrt{\frac{n}{a}} \sum^{j_2}_{j = j_1} \frac{w_j}{2} \Big(\log_2 \lambda_{p-r+q}\big({\mathbf W}(a2^j)\big) - \log_2 \lambda_{p-r+q}\big(\bbE {\mathbf W}(a 2^j)\big)\Big)
$$
\begin{equation}\label{e:h^q-hq_three_terms}
+ \sqrt{\frac{n}{a}} \sum^{j_2}_{j=j_1} \frac{w_j}{2} \Big(\log_2 \lambda_{p-r+q}\big(\bbE {\mathbf W}(a 2^j)\big) - \log_2 \xi_q(a2^j)\Big)
+ \sqrt{\frac{n}{a}} \Big( \sum^{j_2}_{j = j_1} \frac{w_j}{2}  \log_2 \xi_q(a2^j)  - h_q \Big).
\end{equation}
Note that by \eqref{e:xi-i0_scales} in Theorem \ref{t:lim_n_a*lambda/a^(2h+1)}, $\xi_q(a2^j) = a^{2h_q+1} \xi_q(2^{j})$. Therefore, by property \eqref{e:sum_wj=0,sum_jwj=1}, the third term in the sum \eqref{e:h^q-hq_three_terms} is zero. By Proposition \ref{p:|lambdaq(EW)-xiq(2^j)|_bound_2}, the second term in the sum \eqref{e:h^q-hq_three_terms} is bounded by
$$
\sqrt{\frac{n}{a}} \sum^{j_2}_{j=j_1} \frac{|w_j |}{2} \frac{C}{a^{ \varpi}}
\leq C'\sqrt{\frac{n}{a^{ 1 + 2 \varpi}}} \rightarrow 0, \quad n \rightarrow \infty,
$$
where the limit is a consequence of condition \eqref{e:p(n),a(n)_conditions}. Therefore, we can rewrite the left-hand side of \eqref{e:h^q_asymptotically_normal} as
$$
\sum^{j_2}_{j = j_1} \frac{2^{j/2 - 1}w_j}{ \log 2} \sqrt{n_{a,j}} \Big(\log \lambda_{p-r+q}\big({\mathbf W}(a2^j)\big) - \log \lambda_{p-r+q}\big(\bbE {\mathbf W}(a2^j)\big)\Big) + o(1),
$$
and the weak limit \eqref{e:h^q_asymptotically_normal} follows from Theorem \ref{t:asympt_normality_lambdap-r+q}. In the limiting variance in \eqref{e:h^q_asymptotically_normal}, the weight matrix $M \in M(r,mr,\bbR)$ is given by
\begin{equation}\label{e:weight_matrix_M}
M = \Big( \frac{2^{j_1/2}w_{j_1}}{ \log 2} I_r ; \hspace{1mm}\frac{2^{j_1+1/2}w_{j_1+1}}{ \log 2} I_r; \hspace{1mm} \hdots \hspace{1mm}; \hspace{1mm} \frac{2^{j_2/2}w_{j_2}}{ \log 2} I_r\Big),
\end{equation}
where $I_r \in M(r,\bbR)$ is an identity matrix and $m=j_2-j_1+1$. $\Box$\\

\noindent {\sc Proof of Theorem \ref{t:r-hat->r}}: Theorem \ref{t:lim_n_a*lambda/a^(2h+1)} and the scaling relation \eqref{e:xi-i0_scales} show that, as $n \rightarrow \infty$,
$$
\frac{\log_2 \lambda_{p-r+q} {\mathbf W}(a2^j)}{\log_2 (a2^j)} \stackrel{\bbP}{\rightarrow} {2 h_{q}+1 >0}, \quad q = 1,\hdots,r.
$$
In addition, by relation \eqref{e:lim_lambda_p-r(W)} in Theorem \ref{t:lim_n_a*lambda/a^(2h+1)}, $\frac{\log_2 \lambda_{p-r} {\mathbf W}(a2^j)}{\log_2 (a2^j)} \stackrel{\bbP}{\rightarrow} 0$ for every $j$.  Thus, for $q = 1,\hdots,r$ and $\kappa > 0$ as in \eqref{e:def_rhat},
\begin{equation}\label{e:Delta_i,i=p-r+1,...,p}
\Delta_{p-r+q}(j_1,j_2) \stackrel{\bbP}{\rightarrow} \sum_{j=j_1}^{j_2} v_j (2h_q+1) = 2h_{q}+1 > \kappa.
\end{equation}
Likewise,
\begin{equation}\label{e:Delta_i,i=1,...,p-r}
\max_{i=1,\ldots,p-r}|\Delta_i(j_1,j_2)| \leq \sum^{j_2}_{j=j_1}|v_j| \cdot \Big|\frac{\log_2 \lambda_{p-r} {\mathbf W}(a2^j)}{\log_2 (a2^j)} \Big| \stackrel{\bbP}{\rightarrow}0.
\end{equation}
Expression \eqref{e:r-hat-->r} is now a consequence of \eqref{e:Delta_i,i=p-r+1,...,p} and \eqref{e:Delta_i,i=1,...,p-r}. $\Box$\\

The following lemma is used in the proof of Proposition \ref{p:|lambdaq(EW)-xiq(2^j)|_bound_2}.
\begin{lemma}\label{l:p-tilde_upk=o(a-w)}
Suppose the assumptions of Theorem \ref{t:h-hat_asympt_normality} hold, as well as condition $(i)$ in the same theorem. Let $P(n)=Q(n)R(n)$ be the $QR$ decomposition of $P(n)$, where $R(n)\in GL(r,\bbR)$ is upper-triangular, and $Q(n)=(\widetilde \p_1(n),\ldots,\widetilde \p_r(n))\in M(p,r,\bbR)$ has orthonormal columns. Also let
\begin{equation}\label{e:sum<pi(n),up-r+q(n)>2-o(a-varpi)}
\p_{r+1}(n),\ldots,\p_p(n)
\end{equation}
be an orthonormal basis for the nullspace of $P^*(n)$ that is orthogonal to $\textnormal{span}\{\p_{1}(n),\ldots,\p_r(n)\}$. Then, expression \eqref{e:1-<p-tilde_k(n),u_p-r+k(n)>^2=O(small)} holds.
\end{lemma}
\begin{proof}
Note that%
\begin{equation}\label{e:span_ptilde=span_p_2}
\textnormal{span}\{\widetilde{{\mathbf p}}_{  1}(n), \hdots, \widetilde{{\mathbf p}}_{  \ell}(n)\}= \textnormal{span}\{{\mathbf p}_{  1}(n), \hdots, {\mathbf p}_{\ell}(n)\}, \quad \ell = 1,\hdots,r.
\end{equation}
Moreover, for ${\mathbf p}_{\ell}(n)$, $\ell=r+1, \hdots,p$, as in \eqref{e:sum<pi(n),up-r+q(n)>2-o(a-varpi)},
$$
\textnormal{span}\{\widetilde{{\mathbf p}}_{1}(n), \hdots, \widetilde{{\mathbf p}}_{r}(n)\} \perp  \textnormal{span}\{{\mathbf p}_{r+1}(n),\hdots,{\mathbf p}_{p}(n)\}.
$$
Note that expression \eqref{e:supR'max(angles*powerlaws)=O_P(1)} of Lemma \ref{l:|<p3,uq(n)>|a(n)^{h_3-h_1}=O(1)} (with $\widetilde \W(a2^j) = \E\W(a2^j)$) shows that for each fixed  $i\in\{2,\ldots,r\}$,
\begin{equation}\label{e:<p_i,u_p-r+k>}
\langle  \p_i(n),\mathbf u_{p-r+\ell}(n)\rangle^2 = O(a^{-2(h_i-h_\ell)})=O(a^{-2\varpi}), \quad \ell=1,\ldots,i-1.
\end{equation}
Hence, from \eqref{e:span_ptilde=span_p_2},
\begin{equation}\label{e:double_sum_lemma_C1}
\sum_{i=2}^{r}\sum_{\ell=1}^{i-1}\langle  \widetilde \p_i(n),\mathbf u_{p-r+\ell}(n)\rangle^2 = O\Big(\sum_{i=2}^{r}\sum_{\ell=1}^{i-1}\langle  \p_i(n),\mathbf u_{p-r+\ell}(n)\rangle^2\Big) = O(a^{-2\varpi}).
\end{equation}
We now proceed inductively to show \eqref{e:1-<p-tilde_k(n),u_p-r+k(n)>^2=O(small)}.   Turning to the case $k=1$, first note that, as a consequence of \eqref{e:double_sum_lemma_C1}, $\sum_{i=2}^r\langle \widetilde \p_i(n),\mathbf u_{p-r+1}(n)\rangle^2= O(a^{-2\varpi}).$
In addition, Lemma \ref{l:sum<pi(n),up-r+q(n)>2-o(a-varpi)_first} shows that $\sum_{i=r+1}^p\langle \p_i(n), \mathbf u_{p-r+1}(n)\rangle^2= O(a^{-(h_1+\frac{1}{2})})=O(a^{-2\varpi})$, implying
$$
1- \langle \widetilde \p_1(n), \mathbf u_{p-r+1}(n)\rangle^2 = \sum_{i=2}^r\langle \widetilde \p_i(n),\mathbf u_{p-r+1}(n)\rangle^2 + \sum_{i=r+1}^p\langle \p_i(n), \mathbf u_{p-r+1}(n)\rangle^2
$$
$$
= O(a^{-2\varpi}) +O(a^{-2\varpi}) =  O(a^{-2\varpi}),
$$
i.e., \eqref{e:1-<p-tilde_k(n),u_p-r+k(n)>^2=O(small)} holds for $k=1$. We now proceed by induction on $k$. It suffices to consider the range $2 \leq k \leq r$. So, suppose we have $1- \langle \widetilde \p_i(n), \mathbf u_{p-r+i}(n)\rangle^2= O(a^{-2\varpi})$, $i=1,\ldots,k$. Therefore, by decomposing $\widetilde \p_i(n)$ in the eigenvector basis,
$$
1- \langle \widetilde \p_i(n), \mathbf u_{p-r+i}(n)\rangle^2=\sum_{\ell\neq p-r+i} \langle \widetilde \p_i(n), \mathbf u_{\ell}(n)\rangle^2 = O(a^{-2\varpi}), \quad i=1,\ldots,k.
$$
In particular,
\begin{equation}\label{e:inner_p-tilde_i(n),u_p-r+(k+1)(n)=small}
\langle \widetilde \p_i(n),\mathbf u_{p-r+(k+1)}(n)\rangle^2 = O(a^{-2\varpi}), \quad i=1,\ldots, k.
\end{equation}
However, Lemma \ref{l:sum<pi(n),up-r+q(n)>2-o(a-varpi)_first} again shows that
\begin{equation}\label{e:sum_range_i=r+1_p_inner_p_i(n),u_p-r+(k+1)(n)^2}
\sum_{i=r+1}^p\langle \p_i(n), \mathbf u_{p-r+(k+1)}(n)\rangle^2=o(a^{-2\varpi}).
\end{equation}
Relations \eqref{e:inner_p-tilde_i(n),u_p-r+(k+1)(n)=small} and \eqref{e:sum_range_i=r+1_p_inner_p_i(n),u_p-r+(k+1)(n)^2} imply that
$$
1- \langle \widetilde \p_{k+1}(n), \mathbf u_{p-r+(k+1)}(n)\rangle^2 $$$$=\sum_{i=1}^k\langle \widetilde \p_{i}(n), \mathbf u_{p-r+(k+1)}(n)\rangle^2 +\sum_{i=k+2}^r\langle \widetilde \p_{i}(n), \mathbf u_{p-r+(k+1)}(n)\rangle^2 + \sum_{i=r+1}^p\langle \p_i(n), \mathbf u_{p-r+(k+1)}(n)\rangle^2
$$
$$
= O(a^{-2\varpi}) + O(a^{-2\varpi})  + o(a^{-2\varpi}) =O(a^{-2\varpi}),
$$
where we used that $\sum_{i=k+2}^r\langle \widetilde \p_{i}(n), \mathbf u_{p-r+(k+1)}(n)\rangle^2 = O\big( \sum_{i=k+2}^r\langle  \p_{i}(n), \mathbf u_{p-r+(k+1)}(n)\rangle^2\big) = O(a^{-2\varpi})$ by expressions \eqref{e:span_ptilde=span_p_2} and \eqref{e:<p_i,u_p-r+k>}. Thus,   $1- \langle \widetilde \p_k(n), {\mathbf u}_{p-r+k}(n)\rangle^2=O(a^{-2\varpi}), \quad k=1,\ldots, r$; i.e., \eqref{e:1-<p-tilde_k(n),u_p-r+k(n)>^2=O(small)} holds. $\Box$\\
\end{proof}

\section{Auxiliary lemmas}

In this section, for the reader's convenience, we recap the statements of some useful lemmas established in Abry et al.\ \cite{abry:boniece:didier:wendt:2022:wavelet_eigenanalysis}. For simplicity, we adjust the statements to the particular case involving deterministic matrices that is relevant for establishing the statements in Section \ref{s:proof_main_results}. In this section, we also use the notation \eqref{e:def_indexsets}--\eqref{e:P(n)PH_equiv_P}, and write $a$ instead of $a(n)$ whenever convenient.

\vspace{0.2cm}

\begin{lemma}\label{l:|<p3,uq(n)>|a(n)^{h_3-h_1}=O(1)} (Abry et al.\ \cite{abry:boniece:didier:wendt:2022:wavelet_eigenanalysis},  Lemma B.3)
Fix $q \in\{1,\hdots,r\}$ and consider the index sets $\mathcal{I}_-,\mathcal{I}_0,$ and $\mathcal{I}_+$ as in \eqref{e:def_indexsets}. Suppose $\mathcal{I}_+ \neq \emptyset$, and let
\begin{equation}\label{e:def_W-tilde}
\widetilde{\W}(a(n)2^j) =P(n)a(n)^{{\mathbf h + \frac{1}{2}I}}\hspace{0.25mm} \B_a(2^j)\hspace{0.25mm} a(n)^{{\mathbf h + \frac{1}{2}I}} P^*(n) + M_n,
\end{equation}
where
\begin{equation}\label{e:def_Mn}
M_n=a(n)^{2h_q+1}\big(\frac{O(1)}{a(n)^{2h_q+1}}+
\frac{P(n)a(n)^{{\mathbf h}}O(1)}{a(n)^{2h_q+1/2}}+\frac{O(1)^*a(n)^{{\mathbf h}}P^*(n)}{a(n)^{2h_q+1/2}}\big).
\end{equation}
In \eqref{e:def_Mn}, the terms $O(1)$ denote any sequence of matrices of appropriate dimension whose norms are bounded. Also, let ${\boldsymbol {\mathfrak u}}_{p-r+q}(n)$ be a unit eigenvector associated with the $(p-r+q)$--th eigenvalue of $\widetilde{\mathbf W}(a(n)2^j)$. Then,
\begin{equation}\label{e:supR'max(angles*powerlaws)=O_P(1)}
\max_{\ell\in \mathcal{I}_+}\{ |\langle {\mathbf p}_{  \ell}(n),{\boldsymbol {\mathfrak u}}_{p-r+q}(n)\rangle| \hspace{0.5mm}a(n)^{h_{\ell}-h_q}\} = O(1), \quad n \rightarrow \infty.
\end{equation}
\end{lemma}

The following lemma is used in the {proof of Lemma \ref{l:p-tilde_upk=o(a-w)}.}%
\begin{lemma}\label{l:sum<pi(n),up-r+q(n)>2-o(a-varpi)_first}
(Abry et al.\ \cite{abry:boniece:didier:wendt:2022:wavelet_eigenanalysis},  Lemma B.4) For each $n$, let
\begin{equation}\label{e:p_ell(n),ell=r+1,...,p}
\p_\ell(n), \quad \ell=r+1,\ldots,p,
\end{equation}
be an orthonormal basis for the nullspace of $P^*(n)$. Let $\varpi$ be as in \eqref{e:varpi_parameter}. Fix any $q \in\{ 1,\hdots,r\}$ and let ${\boldsymbol {\mathfrak u}}_{p-r+q}(n)$ denote a unit eigenvector associated with the $(p-r+q)$--th eigenvalue of $\widetilde{ \mathbf W}(a(n)2^j)$ as given in Lemma \ref{l:|<p3,uq(n)>|a(n)^{h_3-h_1}=O(1)}. Then,
\begin{equation}\label{e:sum<pi(n),up-r+q(n)>2-o(a-varpi)_first}
 \sum^{p}_{i=r+1}\langle {\mathbf p}_{  i}(n), {\boldsymbol {\mathfrak u}}_{p-r+q}(n)\rangle^2 = O\big(a(n)^{-(h_q+\frac{1}{2})}\big).
\end{equation}
\end{lemma}

The following lemma is used in the proofs of Proposition \ref{p:|lambdaq(EW)-xiq(2^j)|_bound_2} and Lemma \ref{l:p-tilde_upk=o(a-w)}.
It is a slight refinement of Lemma B.6 in Abry et al.\ \cite{abry:boniece:didier:wendt:2022:wavelet_eigenanalysis}. The main difference lies in the second statement in \eqref{e:<p,w>=infinitesimal_2}, which provides the rate at which some of the coefficients of the constructed vector (namely, ${\mathbf v}(n')$) shrink to zero as $n' \rightarrow \infty$. For the reader's convenience, we provide a proof of \eqref{e:<p,w>=infinitesimal_2}.
\begin{lemma}\label{l:<p,w>=infinitesimal}
Fix $q\in\{1,\ldots,r\}$, and for $\ell = 1,\hdots,p$, let ${\boldsymbol {\mathfrak u}}_{\ell}(n)$ be a unit eigenvector associated with the $\ell$--th eigenvalue of the matrix $\widetilde {\mathbf W}(a(n)2^j)$ as defined in  Lemma \ref{l:|<p3,uq(n)>|a(n)^{h_3-h_1}=O(1)}. For each $n$, let
\begin{equation}\label{e:def_gamma_n}
\boldsymbol \gamma_\ell(n):=P^*(n){\boldsymbol {\mathfrak u}}_{p-r+\ell}(n),\quad \ell=1,\ldots,r.
\end{equation}
Suppose $n' \in \bbN'\subseteq \bbN$ is a subsequence along which the limits
\begin{equation}\label{e:lim_na_P*(na)u(na)}
\lim_{{n'} \to\infty} {\boldsymbol \gamma}_\ell (n') =: {\boldsymbol \gamma}_\ell, \quad \ell=1,\ldots,r,
\end{equation}
exist. Fix any %
\begin{equation}\label{e:gamma-tilde}
\widetilde{\boldsymbol \gamma} \in \textnormal{span}\{ {\boldsymbol \gamma}_q,\ldots, {\boldsymbol \gamma}_{r-r_3}\}%
\end{equation}
and any vector
\begin{equation}\label{e:x*}
\mathbf{x}_{*} = (x_{*,{r-r_3+1}},\hdots,x_{*,r})^* \in \bbR^{r_3}.
\end{equation}
Then, we can pick a sequence of unit vectors
\begin{equation}\label{e:v_na}
{\mathbf v}(n') := \sum_{\ell=q}^r c_\ell(n) \boldsymbol{\mathfrak u}_{p-r+\ell}(n')\in \textnormal{span}\{\boldsymbol {\mathfrak u}_{p-r+q}(n'),\ldots,\boldsymbol {\mathfrak u}_{p}(n')\}
\end{equation}
satisfying
\begin{equation}\label{e:<p,w>=infinitesimal_2}
 P^*(n') {\mathbf v}(n')  \to \widetilde{{\boldsymbol \gamma}},\quad n' \rightarrow \infty, \qquad \sum_{\ell \in\mathcal I_+}c_\ell^2(n') = O(a(n')^{-2\varpi}).
\end{equation}
In \eqref{e:<p,w>=infinitesimal_2},
\begin{equation}\label{e:<p,w>=infinitesimal_3}
\langle \mathbf{p}_\ell(n'), {\mathbf v}(n') \rangle = \frac{x_{*,\ell}}{a(n')^{h_\ell-h_q}},\quad \ell\in \mathcal{I}_+.
\end{equation}
\end{lemma}
\begin{proof}
We seek only to establish the second statement in \eqref{e:<p,w>=infinitesimal_2} since all other statements follow from the same argument (cf.\ Abry et al.\ \cite{abry:boniece:didier:wendt:2022:wavelet_eigenanalysis}, Lemma B.6). For notational simplicity, write $n=n'$. Also define
\begin{equation}\label{e:vartheta_n}
{\boldsymbol \vartheta}_n = \Big(\frac{x_{*,\ell}}{a(n)^{h_\ell-h_q}}\Big)_{\ell\in \mathcal{I}_+}\in \bbR^{r_3},
\end{equation}
which is the vector that contains the right-hand terms in \eqref{e:<p,w>=infinitesimal_3}. Note that
\begin{equation}\label{e:vartheta_n->0}
{\boldsymbol \vartheta}_n \rightarrow {\mathbf 0}, \quad n \rightarrow \infty.
\end{equation}
Consider the matrices $\Gamma_n\in M(r,\bbR)$ and $\Gamma_n^+ \in M(r_3,\bbR)$ given by
\begin{equation}\label{e:Xi_n}
\Gamma_n =\big(\langle\mathbf{p}_{\ell}(n), {\boldsymbol {\mathfrak u}}_{p-r+\ell'}(n)\rangle\big)_{1\leq \ell,\ell'\leq r},\quad \Gamma^+_n =\big(\langle\mathbf{p}_{\ell}(n), {\boldsymbol {\mathfrak u}}_{p-r+\ell'}(n)\rangle\big)_{\ell,\ell'\in\mathcal{I}_+}.
\end{equation}
By \eqref{e:lim_na_P*(na)u(na)}, as $n \rightarrow \infty$,
\begin{equation}\label{e:Gamma_n->Gamma}
\Gamma_n \to \Gamma :=({\boldsymbol \gamma}_1,\ldots,{\boldsymbol \gamma}_r).
\end{equation}
In particular, $\Gamma_n^+ \to \Gamma^+=\big( \Gamma_{\ell,\ell'} \big)_{\ell,\ell'\in\mathcal{I}_+}$. By Lemma B.5 in Abry et al.\ \cite{abry:boniece:didier:wendt:2022:wavelet_eigenanalysis}, $\Gamma^+$ is nonsingular. This implies that $\Gamma_n^+$ has full rank for all large $n$, without loss of generality assumed so for all $n$. Therefore, for each fixed $0 \leq s \leq 1$ and for the arbitrary scalars $\alpha_q,...,\alpha_{r-r_3}$ as in \eqref{e:gamma-tilde}, the system of equations
\begin{equation}\label{e:Xi*varsigma=vartheta-vec}
\Gamma_n^+ {\boldsymbol \varsigma} =  {\boldsymbol \vartheta}_n
-
s\cdot \Big( \sum_{\ell=q}^{r-r_3}\alpha_\ell\langle\mathbf{p}_{r-r_3+1}(n), {\boldsymbol {\mathfrak u}}_{p-r+
 \ell}(n)\rangle, \hdots, \sum_{\ell=q}^{r-r_3}\alpha_\ell\langle \mathbf{p}_{r}(n), {\boldsymbol {\mathfrak u}}_{p-r+\ell}(n) \rangle\Big)^*
\end{equation}
has a unique solution
\begin{equation}\label{e:varsigma_n(s)}
{\boldsymbol \varsigma}={\boldsymbol \varsigma}_n(s)= \big( \varsigma_{r-r_3+1,n}(s),\ldots, \varsigma_{r,n}(s) \big)^* \in \bbR^{r_3}.
\end{equation}
Note that, by Lemma \ref{l:|<p3,uq(n)>|a(n)^{h_3-h_1}=O(1)}, as $n \rightarrow \infty$,
\begin{equation}\label{e:<p_ell',u_p-r+ell>->0}
\langle \p_{\ell'}(n), {\boldsymbol {\mathfrak u}}_{p-r+\ell}(n)\rangle \to 0, \quad \ell' \in\mathcal{I}_+, \quad\ell \in \mathcal I_0.
\end{equation}
In view of \eqref{e:vartheta_n->0} and \eqref{e:<p_ell',u_p-r+ell>->0}, the vector on the right-hand side of \eqref{e:Xi*varsigma=vartheta-vec} tends to zero. Therefore, since $\Gamma_n^+$ has full rank, the solution ${\boldsymbol \varsigma}_n(s)$ to the system \eqref{e:Xi*varsigma=vartheta-vec} satisfies
\begin{equation}\label{e:|varsigma|->0}
\sup_{s \in [0,1]}\|{\boldsymbol \varsigma}_n(s)\| \to  0, \quad n \rightarrow \infty.
\end{equation}
Thus, for any small $\varepsilon > 0$, $0 \leq \sup_{s \in [0,1]}\|{\boldsymbol \varsigma}_n(s)\| < \varepsilon$ for large enough $n$. Now note that, for each $n$, the function $s\mapsto f(s) = 1-s^2 - \|{\boldsymbol \varsigma}_n(s)\|^2$ depends continuously on $s$. Moreover, $f(0) > 1 - \varepsilon$ and $f(1) = - \|{\boldsymbol \varsigma}_n(1)\|^2$. Hence, $f$ must have a root
\begin{equation}\label{e:s*(n)_in_(0,1]}
s_*(n)\in(0,1]
\end{equation}
with probability tending to 1. So, for one such root $s_*(n)\in(0,1]$, we can use expression \eqref{e:Xi*varsigma=vartheta-vec} to define the coefficients $c_\ell(n)$, $\ell\in \mathcal I_+$, by means of
\begin{equation}\label{e:varsigma_n=Gamma^+_n-inv*vec_0}
\begin{pmatrix} c_{r-r_3+1}(n) \\ \vdots \\ c_{r}(n)\end{pmatrix} = {\boldsymbol \varsigma}_n(s_*(n)) = ({\Gamma}^+_{n})^{-1}\left( {\boldsymbol \vartheta}_n
-
s_*(n) \cdot \left(\begin{matrix}
 \sum_{\ell=q}^{r-r_3}\alpha_\ell\langle\mathbf{p}_{r-r_3+1}(n), \boldsymbol {\mathfrak u}_{p-r+\ell}(n) \rangle  \\
\vdots \\
\sum_{\ell=q}^{r-r_3}\alpha_\ell\langle \mathbf{p}_{r}(n), \boldsymbol {\mathfrak u}_{p-r+\ell}(n) \rangle
\end{matrix}\right)\right).
\end{equation}
(cf. expression (B.51) in Abry et al.~\cite{abry:boniece:didier:wendt:2022:wavelet_eigenanalysis}). Further define $c_1(n)=\ldots=c_{r_1}(n)=0$, and writing $\widetilde{\boldsymbol \gamma} = \alpha_q \boldsymbol \gamma_q + \ldots + \alpha_r\boldsymbol \gamma_r$, define $c_\ell(n)=\alpha_\ell s_*(n)$ for $\ell \in \mathcal I_0$ (cf.\ (B.49) in Abry et al.\ \cite{abry:boniece:didier:wendt:2022:wavelet_eigenanalysis}). In particular, the resulting vector ${\mathbf c}(n) = (c_1(n),\hdots,c_r(n))^*$ of coefficients appearing in \eqref{e:v_na} is a unit vector. For the purpose of establishing the second statement in \eqref{e:<p,w>=infinitesimal_2}, we focus on studying the decay rate of the entries $c_\ell(n)$, $\ell \notin \mathcal I_+$. By construction, $\|{\boldsymbol \vartheta}_n\| =O(a^{-\varpi})$. In addition, for $\ell \in \mathcal I_0$ and $i\in \mathcal I_+$, $\langle \p_i(n),\boldsymbol{\mathfrak u}_{p-r+\ell}(n) \rangle = O(a^{h_i-h_\ell})=O(a^{-\varpi})$ due to Lemma \ref{l:|<p3,uq(n)>|a(n)^{h_3-h_1}=O(1)}. Moreover, \eqref{e:|varsigma|->0} further implies that the root \eqref{e:s*(n)_in_(0,1]} satisfies $s_*(n) \to 1$ as $n \rightarrow \infty$. Thus, from \eqref{e:varsigma_n=Gamma^+_n-inv*vec_0}, we obtain
$$
 \sum_{\ell \in\mathcal I_+} c_\ell(n)^2 \leq \| (\Gamma_n^+)^{-1}\|^2 \hspace{0.5mm}\| O(a^{-\varpi}) + O(a^{-\varpi})\|^2 = O(a^{-2\varpi}),
$$
as was to be shown. $\Box$\\
\end{proof}

The following corollary and proposition are used in the proof of Proposition \ref{p:|lambdaq(EW)-xiq(2^j)|_bound_2}.
\begin{corollary}\label{c:PWXP^*+R_asymptotics}
(a consequence of Corollary B.1 in Abry et al.\ \cite{abry:boniece:didier:wendt:2022:wavelet_eigenanalysis}) Let $\widetilde{ \mathbf W}(a(n)2^j)$ be as given in Lemma \ref{l:|<p3,uq(n)>|a(n)^{h_3-h_1}=O(1)}.  Then, $\lambda_{p-r+\ell}\Big(\frac{\widetilde{\W}(a(n)2^j)}{a(n)^{2h_\ell+1}}\Big) \to \xi_{\ell}(2^j),$ $\ell=1,\ldots,r$. In particular,
\begin{equation}\label{e:lambda_p-r+q->xi_based_on_residual_Rq(n)_Ba}
\lambda_{p-r+\ell}\Big(\frac{\widetilde{\W}(a(n)2^j)}{a(n)^{2h_q+1}}\Big) \rightarrow \begin{cases}
0, & \ell \in \mathcal I_-;\\
 \xi_{q}(2^j), & \ell \in \mathcal I_0;\\
 \infty, & \ell \in \mathcal I_+,
\end{cases}
\quad n \rightarrow \infty.
\end{equation}
In \eqref{e:lambda_p-r+q->xi_based_on_residual_Rq(n)_Ba}, $\xi_{q}(2^j)$ are the functions \eqref{e:lim_n_a*lambda/a^(2h+1)} appearing in Theorem \ref{t:lim_n_a*lambda/a^(2h+1)}.
\end{corollary}
\begin{proposition}(a consequence of Proposition B.1 in Abry et al. \cite{abry:boniece:didier:wendt:2022:wavelet_eigenanalysis}) \label{p:|lambdaq(EW)-xiq(2^j)|_bound}
Fix $j \in \bbN$ and suppose conditions $(W1-W4)$ and $(A1-A5)$ hold. Further assume that either $(i)$ $0<h_1<\ldots<h_r<1$; or
$(ii)$ $h_1=\ldots=h_r$ and the functions $\xi_q(2^j)$ in \eqref{e:lim_n_a*lambda/a^(2h+1)} satisfy
\begin{equation}\label{e:xiq_distinct_2}
q_1\neq q_2 \Rightarrow \xi_{q_1}(1)\neq \xi_{q_2}(1).
\end{equation}
Let $\widetilde \W(a(n)2^j)$ be as given in Lemma \ref{l:|<p3,uq(n)>|a(n)^{h_3-h_1}=O(1)}.  Then, for each $q\in\{1,\ldots,r\}$, there is a sequence of $(p-r+q)$--th unit eigenvectors $\{{\boldsymbol{\mathfrak u}}_{p-r+q}(n)\}_{n\in\bbN}$ of $\widetilde{\mathbf W}(a(n)2^j)$ along which the limits
\begin{equation}\label{e:<p,u>_to_gamma}
\lim_{n\to\infty} P^*(n) {\boldsymbol{\mathfrak u}}_{p-r+q}(n)=: \boldsymbol{\gamma}_{q}
\end{equation}
exist. %
\end{proposition}

\section{Conflicts of interest}
On behalf of all authors, the corresponding author states that there is no conflict of interest.
\section{Data availability statement}
The datasets generated during and/or analyzed during the current study are available from the corresponding author on reasonable request.

\bibliographystyle{agsm}

\bibliography{highdim}

\end{document}